\documentclass[journal,twocolumn,10pt,letter]{IEEEtran}

\IEEEoverridecommandlockouts
\ifCLASSINFOpdf
\else
\fi

\usepackage{bbold}
\usepackage{amsfonts}
\usepackage{nonfloat}
\usepackage{amssymb}
\usepackage[cmex10]{amsmath}
\usepackage{amsthm}
\usepackage{amssymb}
\usepackage{mathrsfs}
\usepackage{amsbsy}
\usepackage{setspace}
\usepackage{graphicx}
\usepackage{newclude}
\usepackage{latexsym}
\usepackage{lettrine} 

\usepackage{flushend}

\usepackage{authblk}
\usepackage{subfigure}

\usepackage{scalerel,stackengine}
\stackMath

\usepackage{amsmath}

\usepackage{enumitem}

\usepackage{amsthm}

\newtheorem{theorem}{Theorem}

\newtheoremstyle{named}{}{}{\itshape}{}{\bfseries}{.}{.5em}{\thmnote{#3's }#1}
\theoremstyle{named}

\usepackage{amsthm}

\theoremstyle{definition}

\usepackage[table]{xcolor}
\usepackage{multirow}
\usepackage{rotating}
\usepackage{booktabs}
\usepackage{bbm} 
\usepackage{amsmath,epsfig,cite,amsfonts,psfrag}
\usepackage{lipsum}
\usepackage{cuted}
\usepackage{color}

\usepackage{epstopdf}
\usepackage{float}
\usepackage{mathtools}
\usepackage{bbm}
\usepackage{atbegshi}
\usepackage{balance}
\usepackage[utf8]{inputenc}
\usepackage{multirow}

\renewcommand{\arraystretch}{1.2}

\newtheorem{corollary}[theorem]{Corollary}

\usepackage{accents}
\newlength{\dhatheight}

\allowdisplaybreaks

\graphicspath{{figures/}}
\allowdisplaybreaks

\makeatletter
\floatstyle{plain}
\newfloat{twocolequfloat}{b}{zzz}
\floatname{twocolequfloat}{Equation}
\newtheorem{Theorem}{Theorem}

\makeatother

\begin{document}
%







\title{\fontsize{24pt}{29pt}\selectfont Communication Security via Temporal Dependency}




%
%
%

\author{Mohsen~Abedi,
        Ahmed~Badawy, and Amr~Mohamed
\thanks{M. Abedi, A. Badawy, and A. Mohamed are with the Dept. Computer Engineering, Qatar University, Doha, Qatar. Email: \{mohsen.abedi;\,badawy;\,amrm\}@qu.edu.qa.} \vspace{-0.4cm}
}

\maketitle

\thispagestyle{empty}  


\begin{abstract}
Communication security has traditionally been built upon one of two external resources: shared secret keys or a communication advantage over the eavesdropper. However, many practical wireless scenarios, including infrastructure-less, emergency, and highly dynamic networks, cannot guarantee either resource, motivating the need for a new communication-security principle. This paper introduces a new communication-security paradigm that exploits temporal dependency as a security resource. Unlike conventional secrecy techniques that prevent an eavesdropper from recovering transmitted bits, the proposed paradigm allows packet decoding but prevents correct interpretation by making original and dummy packets computationally indistinguishable. Rather than protecting individual transmissions, successive transmissions are intentionally coupled so that future communication depends on correctly interpreting previous ones. As one realization, we develop a state-chained random linear network coding (RLNC) framework in which the synchronization state required to interpret each transmission block is embedded in the previous block. Therefore, synchronization failures propagate across future transmissions, resulting in persistent eavesdropper asynchronization. We analytically characterize the probability and persistence of eavesdropper asynchronization, together with the computational complexity of resynchronization, and develop transmission strategies based on transmit-power and intentional-interference optimization. Numerical results demonstrate sub-second eavesdropper asynchronization under a worst-case adversarial model with no channel advantage, no secret assumptions, complete protocol knowledge, and an arbitrarily stronger eavesdropper.
\end{abstract}


%
\IEEEpeerreviewmaketitle

\section{Introduction}



Conventional cryptographic security and physical-layer security are the two dominant paradigms for protecting wireless communications. Despite their different foundations, both derive security from external resources: cryptographic security relies on shared secret keys~\cite{rivest1978method,wyner1975wire}, whereas physical-layer security exploits a communication advantage through channel asymmetry, CSI, beamforming, artificial noise, cooperative jamming, and related techniques~\cite{loyka2022secrecy,niu2025survey,luo2022secure,zadeh2025encryption,hayashi2023non,Liang2008Secure}. However, both rely on assumptions that may not hold in highly dynamic wireless environments, including secure key management, trusted infrastructure, or a communication advantage over the eavesdropper. As wireless networks evolve toward infrastructure-less and highly heterogeneous deployments, these assumptions become increasingly difficult to guarantee. This motivates the need for a fundamentally different communication-security principle that derives protection from the communication process itself rather than from external resources.

Although cryptographic security and physical-layer security derive protection from different resources, they share a common security principle: each transmission is treated as an independent security event. Accordingly, the compromise or successful interception of one transmission has essentially no influence on the security of future communications.  This paper departs from this transmission-centric view by introducing a communication-security paradigm that exploits temporal dependency as a new security resource. Rather than protecting transmissions independently, the proposed paradigm deliberately couples successive transmissions, making the successful interpretation of future communications contingent upon correctly interpreting previous ones. Therefore, a synchronization failure propagates across future transmissions until synchronization is restored.

\begin{figure*}
    \centering
    \includegraphics[width=0.8\linewidth]{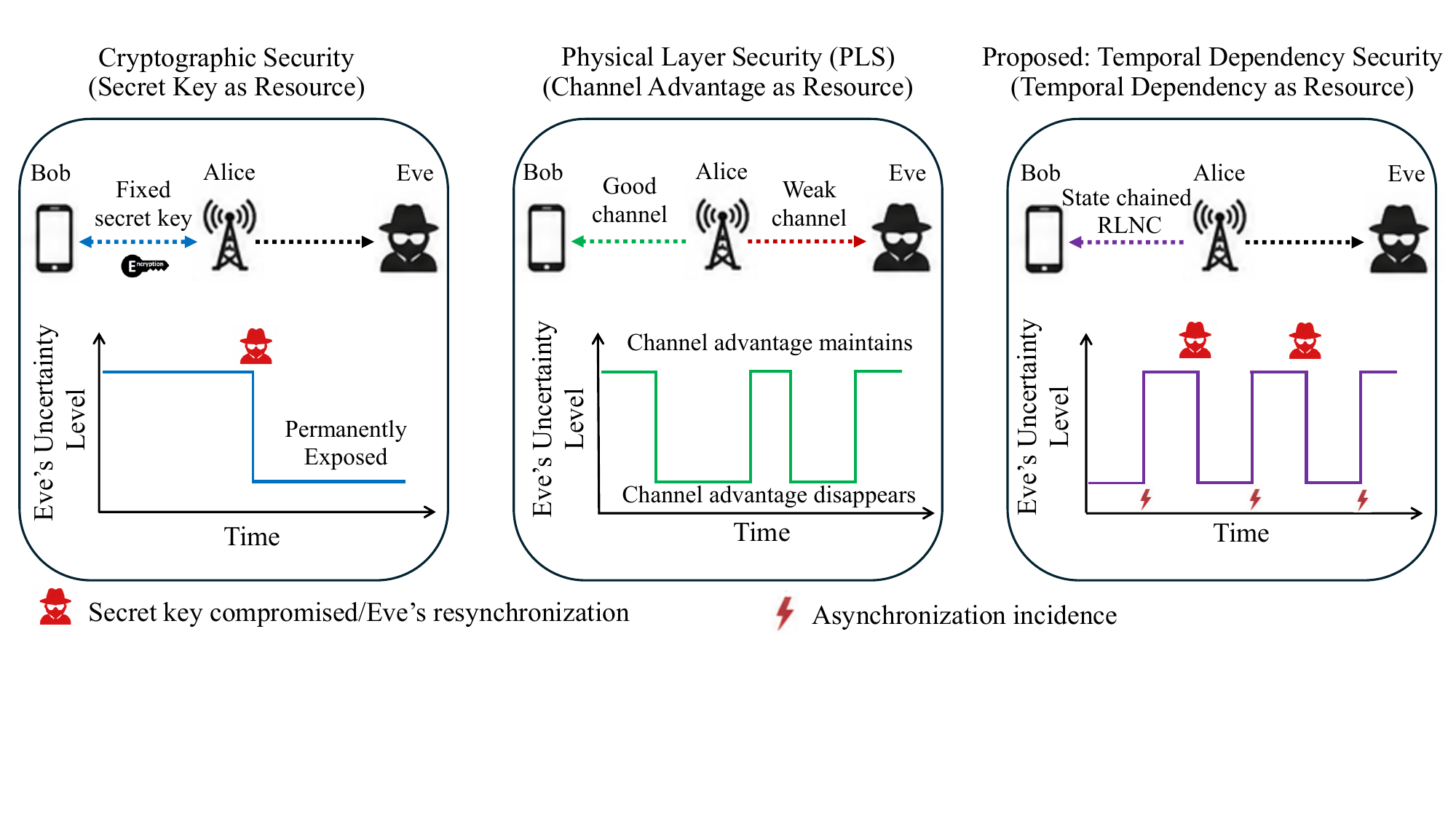}
    \vspace{-2cm}
    \caption{Comparison of security paradigms: Cryptographic security, physical layer security, and the proposed temporal dependency security. Unlike cryptographic security, where a key compromise permanently exposes future communications, and conventional physical-layer security, which requires a channel advantage, the proposed paradigm repeatedly re-establishes uncertainty in packet interpretation through eavesdropper asynchronization without relying on either assumption.}
    \label{fig:crypto_compare}
    \vspace{-0.4cm}
\end{figure*}

The proposed communication-security paradigm can be realized using any communication framework capable of creating temporal dependencies across successive transmissions. In this paper, we demonstrate one such realization using Random Linear Network Coding (RLNC), whose linear combinations naturally enable the information required to interpret future transmissions, referred to hereafter as the \emph{synchronization state}, to be embedded within previous transmissions. Specifically, a chained synchronization state is conveyed across consecutive transmission blocks, making recovery of the synchronization state necessary for the correct interpretation of subsequent communications. Unlike existing RLNC-based security techniques, which primarily protect the current coding generation~\cite{tassi2019intercept,chatzigeorgiou2022impact}, the proposed state-chained RLNC framework realizes the temporal-dependency principle. While the underlying security paradigm is general, the analytical and numerical results presented in this paper are specific to this RLNC realization.

The key advantage of exploiting temporal dependency is that a single synchronization failure can affect the interpretation of many future transmissions. Once the eavesdropper loses synchronization, subsequent transmissions remain uninterpretable until synchronization is restored, while the legitimate receiver remains synchronized. Hence, successful decoding no longer guarantees successful interpretation, which depends on recovering the evolving synchronization state. This property is particularly attractive for communication scenarios in which temporary compromise is difficult to avoid, such as satellite communications, UAV swarms, and infrastructure-less emergency networks. Unlike conventional stateful communication protocols, where state evolution primarily supports synchronization, authentication, or forward secrecy~\cite{yu2024multi,xiong2024conditional}, the proposed framework deliberately exploits it as the security mechanism itself, making state recovery a computational requirement rather than merely a protocol function.


The proposed framework is designed and evaluated under a deliberately conservative threat model that eliminates nearly all conventional sources of communication security. The transmitter has no knowledge of either receiver's channel, while the eavesdropper is assumed to be passive and possess complete knowledge of the communication system, including the protocols, coding procedures, acknowledgments, and even the cryptographic keys. Furthermore, the eavesdropper may have a substantially stronger channel than the legitimate receiver due to a larger antenna aperture, more receive antennas, or a more favorable location.

 Based on this conservative threat model, the main contributions of this paper are summarized as follows:
\begin{itemize}
\item We propose a communication-security paradigm based on temporal dependency, whereby the successful interpretation of future transmissions is intentionally coupled to the correct interpretation of previous ones. As one realization, we develop a state-chained RLNC transmission scheme that transforms a single eavesdropper synchronization failure into persistent asynchronization across future transmission blocks.

\item We analytically characterize how temporal dependency propagates synchronization failures, the expected time to its first occurrence, and the computational complexity of the eavesdropper's resynchronization.

\item We develop optimization strategies that achieve sub-second eavesdropper asynchronization even against an arbitrarily stronger eavesdropper, demonstrating the practical feasibility of the proposed paradigm.

\end{itemize}

Unlike the existing communication-security techniques, the proposed paradigm does not attempt to hide the transmitted packets themselves. Instead, it allows the eavesdropper to decode the received packets while preventing their correct interpretation by embedding the transmitted information among computationally indistinguishable original and dummy packets. As a result, communication security shifts from protecting packet recovery to protecting packet interpretation. Further, unlike many practical physical-layer security techniques that rely on receiver feedback (e.g., CSI) and cryptographic systems that require pre-established or interactively negotiated secret keys, the proposed framework requires only sparse acknowledgment bits that are already used for reliable communication.

Fig.~\ref{fig:crypto_compare} contrasts the proposed paradigm with cryptographic and physical-layer security~\cite{diffie2022new,clark2013sok,khan2023survey,shor1994algorithms}. It illustrates that temporal dependency establishes security by repeatedly inducing eavesdropper asynchronization rather than relying on secret credentials or channel advantage. Accordingly, temporal dependency preserves communication security even when secret keys are compromised and channel advantage no longer exists. To the best of our knowledge, existing communication systems employ temporal dependency primarily for synchronization, reliability, or protocol continuity rather than as the mechanism that provides communication security itself. 



The remainder of this paper is organized as follows. Section~II introduces the considered system model and communication assumptions. Section~III presents the proposed state-chained RLNC realization of the temporal-dependency communication-security paradigm. Section~IV provides the analytical characterization of the asynchronization and the computational complexity imposed on the eavesdropper. Section~V formulates the channel model, system metrics, and the corresponding optimization problems. Section~VI presents numerical results that validate the analytical framework and illustrate the trade-offs between communication performance and security. Finally, Section~VII concludes the paper.



\section{System and Communication Model}
In the considered system model, illustrated in Fig. \ref{fig:system_model}, Alice (A) transmits RLNC-encoded packets to Bob (B) over a point-to-point (PtP) link, while Eve (E) passively intercepts the transmissions. Consistent with the worst-case threat model introduced in Section I, Alice has no knowledge of the CSI of either Bob or Eve, while Eve is assumed to possess complete knowledge of the communication framework, including protocols, coding procedures, packet formats, acknowledgments, and even the CSI between Alice and Bob. Moreover, Eve  remains entirely silent throughout the communication process. This section introduces the baseline RLNC communication model used throughout the paper.

\begin{table}[t]
\centering
\caption{Notations and symbols}
\small
\renewcommand{\arraystretch}{0.95}
\setlength{\tabcolsep}{8pt}
\begin{tabular}{ll}
\hline
\textbf{Symbol} & \textbf{Description} \\
\hline
$\mathbf{a}_t^{(b)}$ & Coding  vector \\
$\alpha$ & Intentional interference coefficient \\
$\mathbf{B}_t/\mathbf{E}_t$ & Bob's$/$Eve's decoding matrix at time slot $t$ \\
$\textbf{d}^{(b)}/\textbf{x}^{(b)}$ & Vector of dummy$/$original packets  \\
$\text{dim}(.)$ & Matrix dimension\\
$\epsilon(.)$ & Interpretation uncertainty level\\
$\text{f}_\text{ACK}$(.) & PDF of  acknowledgment delay\\
$\text{f}_\text{B}$(.) & PDF of the time that Bob achieves full rank \\
$\text{f}_\text{A}$(.) & PDF of the time that Alice receives the ACK \\
$\ell(.)/\eta(.)$ & Packet latency$/$throughput \\
$\mathrm{m}(.)/\mathrm{r}(.)$ & Number of rows$/$rank of a matrix \\
$N/D$ & Number of packets$/$dummy packets  \\
$\mathcal{P}$(.) & Probability function \\

${\bf{p}}^{(b)}$& Vector of block packets at block $b$\\
$\mathrm{P}_{\text{asy}}$ & Probability that Eve loses the synchronization\\& state required for interpreting the next  block.\\
$\mathrm{P}_{\text{B}}/\mathrm{P}_{\text{E}}$ & Packet success probability at Bob$/$Eve \\
$q_{t}^{(b)}$ & RLNC-encoded packet\\
row(.) & Row space of a matrix\\
$\mathrm{s}^{(b)}$ & State vector of block $b$ \\
$t$ & Time slot \\
$\tau_q/t_b$ & Packet$/$block transmission time\\
$\tau_{\rm asy}$ & Mean time to the first asynchronization incidence\\
$\tau_{\rm cmp}$ & Computational time required to resynchronize \\
$T$ & The time slot that Bob achieves full rank\\
$T'$ & Time slot at which Alice receives the ACK \\
$T_{\text{ack}}$ & Acknowledgment delay \\
\hline
\end{tabular}
\label{tab:notation}
\end{table}

Let \(N\) denote the number of packets in each block, and let
\[
{\bf{p}}^{(b)}=[p_1^{(b)}, p_2^{(b)}, \ldots, p_N^{(b)}]^\top
\]
represent the set of packets in block \(b\), where $
p_i^{(b)} \in {\mathbb{F}}_2^n
$
denotes the \(i^{\text{th}}\) packet of block \(b\) with length \(n\) over GF(2) at Alice for transmission to Bob. Each transmitted packet is generated using RLNC as follows.
\begin{equation}
\begin{aligned}
q_t^{(b)}=\bigoplus_{i=1}^N a_{ti}^{(b)}p_i^{(b)},
\end{aligned}
\label{eq:RLNC-packet}
\end{equation}
where  $a_{ti}^{(b)}\sim$\,Bern($0.5$) is randomly generated at Alice with a symmetric Bernoulli distribution.  Then, the coding vector
\[
{\bf a}_t^{(b)}=[a_{t1}^{(b)}, a_{t2}^{(b)}, \ldots, a_{tN}^{(b)}]
\]
is appended as a header to the RLNC-encoded packet $q_t^{(b)}$
for transmission at time slot \(t\). Since the coding coefficients are binary, the coding vector is represented as a binary vector of length \(N\). Each transmitted packet therefore includes the RLNC-encoded packet, its associated coding vector, and the protocol headers required for transmission and decoding.

At the beginning of block $b$, Bob and Eve independently construct time-varying decoding matrices, denoted by ${\bf B}_t\in\mathbb{F}_2^{m({\bf B}_t)\times N}$ and ${\bf E}_t\in\mathbb{F}_2^{m({\bf E}_t)\times N}$, respectively. Upon the transmission of packet $q_t^{(b)}$ at time slot $t$, Bob and Eve successfully receive it with probabilities $\mathrm{P}_{\rm B}$ and $\mathrm{P}_{\rm E}$, respectively, whose analytical expressions for Rayleigh fading channels are introduced in Section V. 
If the RLNC-encoded packet of block $b$ is successfully received at time slot $t$, the packet $q_t^{(b)}$ is appended as a new row to the packet matrix of the corresponding receiver, namely ${\bf Q}_t^{\rm B}$ for Bob and ${\bf Q}_t^{\rm E}$ for Eve. Likewise, the associated coding vector ${\bf a}_t^{(b)}$ is appended as a new row to the corresponding decoding matrix, i.e., ${\bf B}_t$ for Bob and ${\bf E}_t$ for Eve.

As soon as Bob's decoding matrix becomes full rank, i.e.,
$
{\rm r}({\bf B}_T)=N,
$
at time slot $t=T$, Bob decodes the block packets according to
\begin{equation}
\begin{aligned}
{\bf{p}}^{(b)}={\bf{B}}_T^{-1}{\bf{Q}}_T^{\text{B}}.
\end{aligned}
\label{eq:RLNC-decoding-Bob}
\end{equation}
At the same instant, Bob initiates the acknowledgment (ACK) procedure toward Alice. Let the random variable $T_{\rm ack}$ denote the ACK delay, i.e., the interval between the time slot $T$ at which Bob's decoding matrix becomes full rank and the time slot at which Alice receives the ACK.

\begin{figure}
    \centering
    \includegraphics[width=0.7\linewidth]{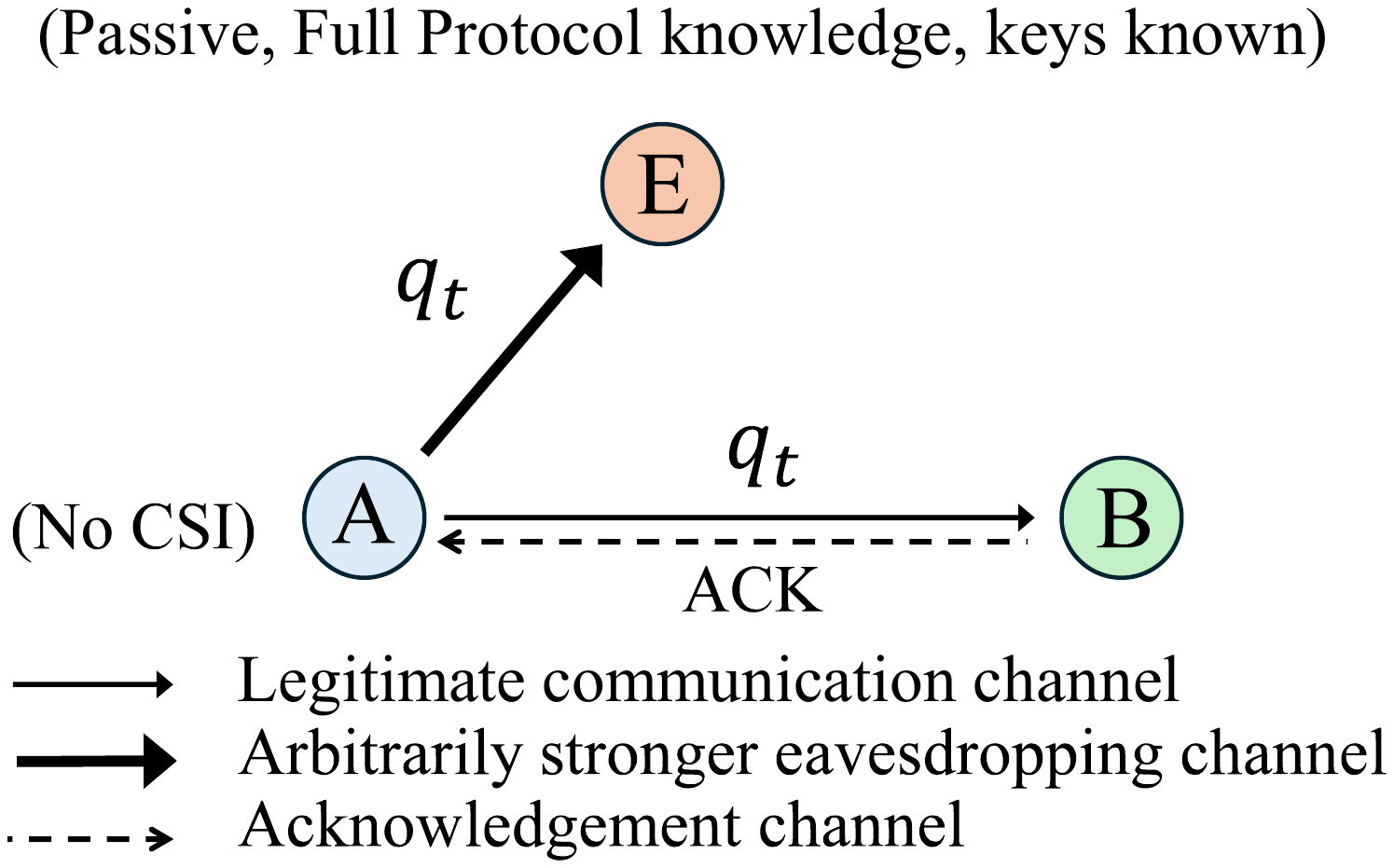}
    \vspace{-0.7cm}
    \caption{Baseline system and communication model. Alice (A) transmits RLNC-encoded packets to Bob (B) over a PtP link without CSI. Passive Eve (E), with full protocol knowledge, intercepts the transmissions and may have a stronger channel, while Bob returns an ACK after successful decoding.}
   
    \label{fig:system_model}
     \vspace{-0.3cm}
\end{figure}

Upon receiving the ACK, Alice stops transmitting linear combinations of the current block at time slot
\begin{equation}
\begin{aligned}
&T'=T+T_{\text{{ack}}}.
\label{eq:T_prime}
\end{aligned}
\end{equation}
During the interval from $T+1$ to $T'$, Bob has already decoded the block, while Alice is unaware of this fact. Therefore, Alice continues transmitting additional RLNC-encoded packets of the current block until the ACK is received at time slot $T'$. These packets do not contribute to Bob's decoding process, since their coding vectors lie in the row space of ${\bf B}_T$. However, they may increase the rank of Eve's decoding matrix.

Then, if Eve achieves full rank by time slot $T'$, i.e., $\mathrm{r}({\bf E}_{T'})=N$, Eve can recover the block packets as
\begin{equation}
\begin{aligned}
{\bf{p}}^{(b)}={\bf{E}}_{T'}^{-1}{\bf{Q}}_{T'}^{\text{E}}.
\end{aligned}
\label{eq:RLNC-decoding-Eve}
\end{equation}
However, whether these decoded packets can be correctly interpreted depends on the synchronization state introduced in Section III.
In contrast, if Eve remains rank deficient, i.e.,
$
\mathrm{r}({\bf E}_{T'})<N,
$
she cannot recover all packets in ${\bf p}^{(b)}$, preventing recovery of the transmitted block. At the end of time slot $T'$, Alice proceeds with the transmission of block $b+1$, regardless of Eve's decoding status.

Throughout the remainder of this paper, we distinguish between several related concepts. 
\emph{Interpretation} refers to correctly identifying which decoded packets correspond to original data and which correspond to dummy packets. \emph{Synchronization} denotes possession of the synchronization state required to interpret the current transmission block. An \emph{asynchronization incidence} occurs when Eve first fails to recover the synchronization state required for the next block. 
Finally, \emph{interpretation uncertainty} quantifies the fraction of decoded packets whose identities remain ambiguous without knowledge of the synchronization state.

\section{Temporal-Dependency Construction via State-Chained RLNC}\label{section:structure}

To realize the proposed temporal-dependency security paradigm, the synchronization state of each block is embedded in the dummy packets of the preceding block, intentionally coupling the interpretation of successive transmission blocks while leaving the RLNC decoding process itself unchanged. Here, the synchronization state of block $b+1$, which determines the dummy-packet locations, is generated during the transmission of block $b$ and secret-shared across its dummy packets. As a result, recovering the state of one block becomes necessary to interpret the next, establishing temporal dependency across successive blocks before RLNC encoding.


\subsection{State-Chained RLNC Construction}

Let the vector of original data packets at block $b$ be 
\begin{equation}
    {\textbf{x}}^{(b)}=[x_1^{(b)}, x_2^{(b)}\ldots,x_{N-D}^{(b)}],
\end{equation}  and  the vector of dummy-packets be 
\begin{equation}
    \textbf{d}^{(b)}=[d_1^{(b)}, d_2^{(b)}\ldots,d_{D}^{(b)}],
\end{equation}where $D$ denotes the number of dummy packets. 

 The state of block $b$ is represented by a binary vector indicating the locations of the dummy packets within the transmission block, expressed as
\begin{equation}
    {\textbf{s}}^{(b)}=[s_1^{(b)},s_2^{(b)},\ldots,  s_N^{(b)}]\in\mathbb{F}_2^{N}, 
\end{equation}
where $s_j^{(b)}=\{0,1\}$ for $1\leq j\leq N$, with the weight of${\text{w}}({\textbf{s}}^{(b)})=D$. 
Using the state, dummy packets are inserted into the vector of original packets as
\begin{equation}
\begin{aligned}
{\textbf{p}}^{(b)}=[p_1^{(b)}, p_2^{(b)},\ldots,p_N^{(b)}]^\top=\operatorname{ins}({\textbf{x}}^{(b)},{\textbf{d}}^{(b)}, {\textbf{s}}^{(b)}),
\end{aligned}\label{eq:insertion}
\end{equation}
where
\[
p_j^{(b)} =
\begin{cases}
d_k^{(b)}, & \text{if } s_j^{(b)}=1,\quad k=1,\ldots,D,\\[2mm]
x_{\iota(j)}^{(b)}, & \text{if }  s_j^{(b)}=0,
\end{cases}
\]
for $1\leq j\leq N$ with 
$\iota(j)=|\{s_i^{(b)}=0: i\leq j\}|
$.

 The state of block $b+1$, denoted by ${\textbf{s}}^{(b+1)}$, is randomly generated at the beginning of block $b$. This state is mapped to the dummy-packet vector through
\begin{equation}
\begin{aligned}
{\textbf{d}}^{(b)}=\mathcal{S}({\textbf{s}}^{(b+1)}), 
\end{aligned}\label{eq:secret_sharing}
\end{equation}
where $\mathcal{S}(.)$ is a reversible secret-sharing mapping of a vector in ${\mathbb{F}}_2^N$ with weight $D$ into a vector of $D$ packets. The secret share mapping should satisfy the following  properties:
\begin{enumerate}[label=(\roman*)]
    
    \item  there must exist a combination function that uniquely recovers the state, described as
    \begin{equation}
    {\textbf{s}}^{(b+1)}=\mathcal{C}({\textbf{d}}^{(b)}),
    \end{equation}
    
    \item  the synchronization state  can be recovered only from the complete set of dummy packets, and finally,
    
    \item the generated dummy packets in ${\textbf{d}}^{(b)}$ should follow a statistically and semantically indistinguishable distribution from the original data packets in ${\textbf{x}}^{(b)}$.
\end{enumerate}
The latter condition specifically prevents Eve from using application-layer validation, statistical inference, or learning-based methods to distinguish any individual dummy packet, or any incomplete subset of dummy packets, from the original data packets. This property is fundamental to the security of the proposed framework, forcing Eve to rely on brute-force search to determine the complete set of dummy packets. 

Fig.~\ref{fig:encoder} summarizes the construction process. At the beginning of block $b$, the synchronization state of block $b+1$ is generated and secret-shared into the dummy packets. These dummy packets are then inserted into the current data block according to the current synchronization state before RLNC encoding, thereby embedding the information required for the next transmission block within the current one.


\subsection{Design of the Dummy Packets}\label{subsection:dummy_distribution}
To prevent Eve from identifying the dummy packets, all packets in the transmitted block should exhibit similar statistical characteristics, protocol headers, and packet signatures. Therefore, even after successfully decoding the entire RLNC block, Eve cannot correctly interpret the decoded packets because she cannot distinguish original packets from dummy packets without knowledge of the synchronization state. As a result, Eve is forced to rely on brute-force search to identify the complete set of dummy packets. Accordingly, each dummy packet should be generated to be statistically and semantically indistinguishable from the original data packets.

As commonly encountered in encrypted or compressed traffic, the original data packets in $\textbf{x}^{(b)}$ 
 are assumed to be  statistically indistinguishable from uniformly random bit sequences, with entropy of $\textbf{H}(x_i^{(b)})=n$ bits for all $i$, where $n$ denotes the number of bits in each original data packet.
 Accordingly, the secret share mapping can be designed as 
\begin{equation}
\begin{aligned}
\textbf{d}^{(b)}=\mathcal{S}(\textbf{s}^{(b+1)})=\Big[&d_1^{(b)}\sim \text{Bern}(0.5)^n,\\&d_2^{(b)}\sim \text{Bern}(0.5)^n,  \\& 
~~~~~~~\vdots
\\& 
d_{D-1}^{(b)}\sim \text{Bern}(0.5)^n, 
\\& 
d_{D}^{(b)}= \text{L}(\textbf{s}^{(b+1)})\oplus \big(\bigoplus_{i=1}^{D-1}d_i^{(b)}\big)\Big],
\label{eq:dummy_generator_entropy_1}
\end{aligned}
\end{equation}
wherein the first $D-1$ dummy packets are independently generated as random binary vectors, whereas the last packet embeds the synchronization state. Further, $\text{Bern}(0.5)^n$ denotes the function that generates a vector of size $n$ with symmetric Bernoulli distribution and  $\text{L}: \mathbb{F}_2^{N}\rightarrow\mathbb{F}_2^n$ can be any invertible mapping with $n>N$.  The construction described in \eqref{eq:dummy_generator_entropy_1} ensures that every dummy packet contributes to state recovery while preserving statistical indistinguishability.

 The conditions (i) and (ii) follow directly from the construction of $\textbf{d}^{(b)}$ via \eqref{eq:dummy_generator_entropy_1}.
With this construction, condition~(iii) is also satisfied, since all dummy packets are randomly generated with entropy
$
\mathbf{H}(d_i^{(b)})=\mathbf{H}(x_j^{(b)})=n
$ bits for all $i$ and $j$.
Conditions~(i) and~(ii) are also satisfied by the definition of $d_D^{(b)}$ in \eqref{eq:dummy_generator_entropy_1}, as the complete set of dummy packets uniquely determines the state through a combination function
\begin{equation}
\begin{aligned}
{\textbf{s}}^{(b+1)}=\mathcal{C}(\textbf{d}^{(b)})=\text{L}^{-1}\big(\bigoplus_{i=1}^D d_i^{(b)}\big),
\label{eq:combinator_entropy_1}
\end{aligned}
\end{equation}
 while any proper subset of the dummy packets is insufficient to recover it. 

\begin{figure}
    \centering
    \includegraphics[width=1\linewidth]{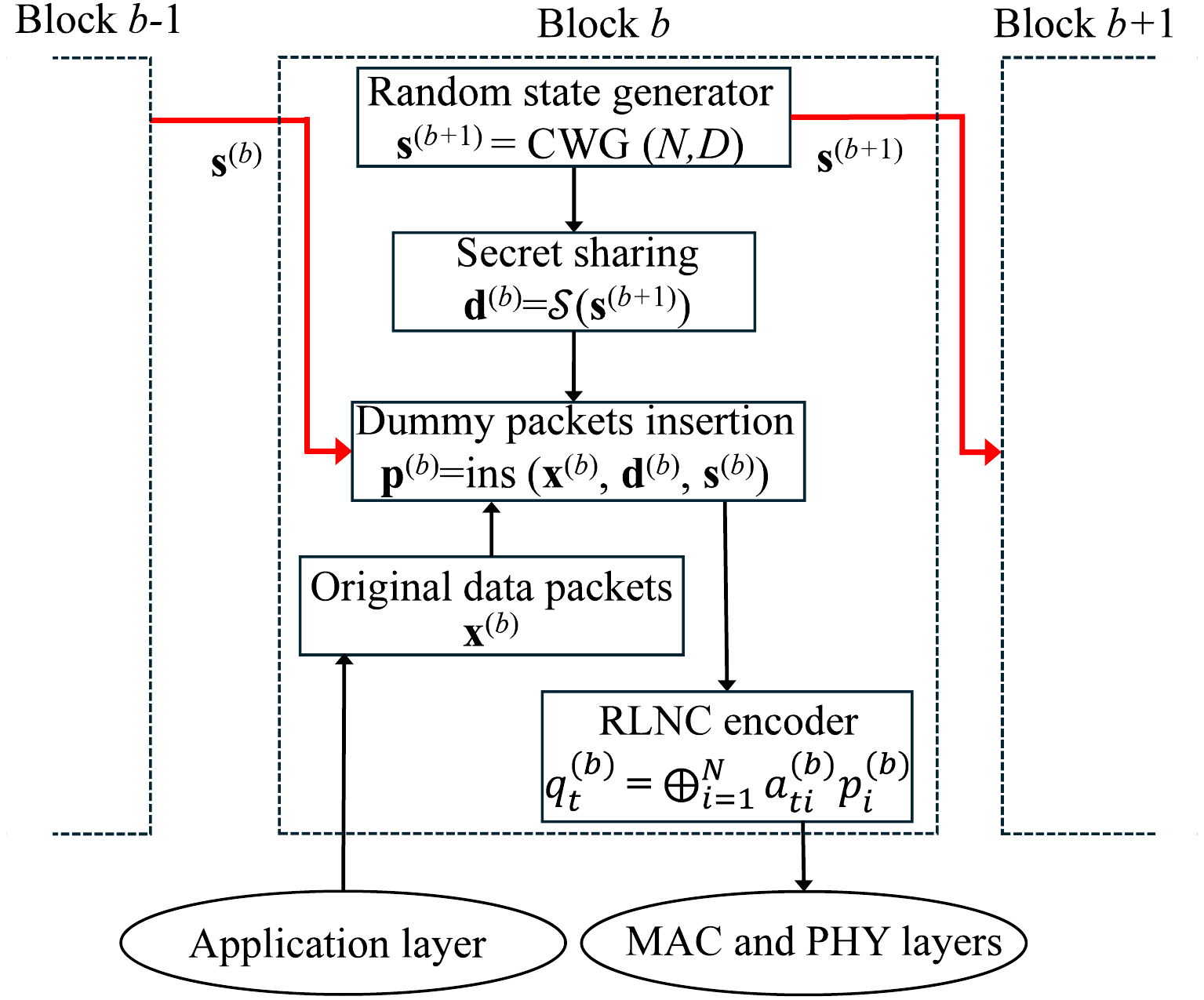}
    \vspace{-0.4cm}
    \caption{ Construction of temporal dependency communication-security paradigm using a state-chained RLNC framework. The synchronization state of each block is secret-shared into the dummy packets of the preceding block, making the correct interpretation of each block depend on recovering the synchronization state embedded in the previous block (red arrows).}
    \vspace{-0.2cm}
    \label{fig:encoder}
\end{figure}

If the original data packets are subject to application-layer validation, the same compression, encryption, packetization, and higher-layer protocol processing can be applied to the dummy packets, producing protocol-compliant packet structures~\cite{wright2009traffic,seo2016discussion}.

\subsection{Asynchronization Incidence and Propagation}

Consider transmission block $b$. Let $\mathbf{E}_{T'}$ denote Eve's decoding matrix at time slot $T'$, when Alice receives Bob's acknowledgment and terminates transmission of the current block. Since Bob is synchronized at the beginning of block $b$, he correctly separates the original and dummy packets after decoding and recovers the synchronization state of block $b+1$ using \eqref{eq:combinator_entropy_1}. Hence, Bob successfully decodes every block and therefore remains synchronized, allowing correct interpretation of every subsequent block.

For Eve, synchronization at block $b$ is maintained whenever $\textbf{s}^{(b)}$
 lies in the row space of her decoding matrix as
\begin{equation}
\mathbf{s}^{(b)} \in \mathrm{row}\!\left(\mathbf{E}_{T'}\right).\end{equation}
 In this case, Eve can form the same linear combination of her decoded packet observations that isolates the combination of dummy packets $\bigoplus_{i=1}^D d_i^{(b)}$, required by the combination function expressed in  \eqref{eq:combinator_entropy_1}. Otherwise, that combination is not identifiable from her available observations.
In contrast, an asynchronization incidence occurs when
\begin{equation}
\textbf{s}^{(b)} \notin \mathrm{row}\!\left(\mathbf{E}_{T'}\right),
\label{eq:async_condition}
\end{equation}
in which Eve cannot recover the dummy-packet combination or determine the next state $\textbf{s}^{(b+1)}$. The probability of this event is denoted by $\mathrm{P}_{\rm asy}$ and is characterized in Section~IV.

Once the synchronization state $\textbf{s}^{(b+1)}$ is lost, Eve may continue decoding subsequent RLNC blocks, starting from block $b+1$, whenever sufficient coded packets are received. However, these decoded blocks cannot be correctly interpreted because the dummy-packet locations remain unknown.
 As a result, she cannot recover the synchronization state of block $b+2$ via
\begin{equation}
\mathbf{s}^{(b+2)}
=
\mathcal{C}\!\left(\mathbf{d}^{(b+1)}\right),
\end{equation}
causing eavesdropper asynchronization to propagate indefinitely across future blocks.

To resynchronize, Eve must correctly identify the dummy packets of at least one asynchronous full-rank block. Since the dummy packets are statistically indistinguishable from the original packets, as discussed in Section~III-B, they cannot be identified using protocol signatures or statistical features. Therefore, Eve must exhaustively search over all combinations of candidate dummy-packet subsets. Once the correct subset is identified for block $b+i$, Eve obtains
\begin{equation}
\mathbf{s}^{(b+i+1)}
=
\mathcal{C}\!\left(\mathbf{d}^{(b+i)}\right),
\end{equation}
thereby restoring synchronization from block $b+i+1$ onward.

Unlike conventional cryptographic systems, the proposed framework does not rely on a secret bootstrap phase. The synchronization state of the first transmission block, $\textbf{s}
^{(1)}$, is assumed to be public and may therefore be recovered by Eve. Consequently, confidentiality of the first communication block is not guaranteed. However, once Eve fails to recover the synchronization state of a subsequent block because insufficient RLNC observations prevent recovery of the dummy-packet combination, all future synchronization states become unavailable, since each new state is embedded only in the previous block. Thereafter, the temporal dependency becomes self-sustaining, with each synchronization state generated and conveyed through the communication itself, eliminating the need for any pre-shared or subsequent key-refresh procedures.


\section{Computational Analysis}
The time at which Bob achieves full rank determines when he transmits the acknowledgment, thereby defining the interval during which Alice continues transmitting additional RLNC packets and Eve continues collecting observations. Accordingly, the probability distribution of Bob's full-rank decoding time underlies the characterization of Eve's observation window and the probability of asynchronization, and is described by the following theorem.
\begin{Theorem}
\label{Lemma_Bob_Ack_time_pdf}
The probability that the decoding matrix of Bob becomes full rank  at time slot $t=T$ is denoted as $\text{f}_\text{B}(t=T)$, and described as
\begin{equation}
\begin{aligned}
&\text{f}_\text{B}(t=T)=\mathcal{P}\Big(\text{r}({\bf{B}}_T)=N,\text{r}({\bf{B}}_{T-1})=N-1\Big)  \\& 
=\sum_{M=N}^T \text{F}_1(M,N)\text{F}_2(T,M,P_\text{B})\text{F}_3(M,N)
\label{eq:Bob_Ack_time_pdf}
\end{aligned}
\end{equation}
 where
\begin{equation*}
\begin{aligned}
&\text{F}_1(M,N)=\dfrac{2^{-(M-2)N-1}}{2^N-1},\\
& \text{F}_2(T,M,\mathrm{P}_\text{B})=\binom{T-1}{M-1} \mathrm{P}_\text{B}^{M} (1-\mathrm{P}_\text{B})^{T-M},
\label{eq:defining_F1_F2}
\end{aligned}
\end{equation*}
and
\begin{equation*}
\begin{aligned}
\text{F}_3(M,N)=\prod_{i=0}^{N-2} \dfrac{(2^{M-1}-2^i)(2^N-2^i)}{2^{N-1}-2^i}
\label{eq:Defining_F3}
\end{aligned}
\end{equation*}
for $T\geq N$ and $\text{f}_\text{B}(t=T)=0$ otherwise.

\end{Theorem}
\begin{proof}
The proof is provided in Appendix \ref{Appendix_A}.
\end{proof}
 
Since $T'$
 is the sum of the decoding time $T$ and the ACK delay, its PDF is obtained by convolution as
 \begin{equation}
\begin{aligned}
&{\text{f}}_\text{A}(t)={\text{f}}_\text{B}(t)*{\text{f}}_{\text{ACK}}(t),
\label{eq:T_prime_pdf}
\end{aligned}
\end{equation}
where ${\text{ f}}_{\text{ACK}}(t=T_{\text{ack}})$ refers to the PDF of the ACK delay.

An asynchronization incidence occurs whenever the dummy-packet state vector does not belong to the row space of Eve's final decoding matrix, i.e., $\textbf{s}\notin \text{row}(\textbf{E}_{T'})$. In this case, Eve cannot recover the dummy-packet combination or determine the state of the next block. The resulting probability of asynchronization depends on three random mechanisms: Bob's decoding time, the ACK delay, and Eve's packet reception. The expression is obtained by averaging the conditional asynchronization probability over Bob's decoding time, the ACK delay, and Eve's packet receptions before and after Bob achieves full rank.

\begin{Theorem}\label{Lemma:P_async}
The probability of asynchronization incidence for Eve in a block is given as
\begin{equation}
\begin{aligned}
&\text{P}_{\rm asy}=\mathcal{P}\!\left(
{\bf s}\notin {\rm row}({\bf E}_{T'})
\right)
\\&
=\sum_{T'=N}^{\infty}
\sum_{T=N}^{T'}
\text{f}_{\rm ACK}(T'-T)
\sum_{M=N}^{T}
\text{f}_{\text{B},M}(T,M)
\text{P}_{{\rm asy}|T',T,M},
\end{aligned}
\label{eq:theorem_Pasy_exact_G}
\end{equation}
where
\begin{equation}
\begin{aligned}
\text{f}_{\text{B},M}&(T,M)
=
2^{-(M-1)N}
\Phi(M-1,N,N-1)
\\
& \times
\binom{T-1}{M-1}
\text{P}_\text{B}^{M-1}(1-\text{P}_\text{B})^{T-M}
\text{P}_\text{B}
\frac{2^N-2^{N-1}}{2^N-1},
\end{aligned}
\label{eq:theorem_fBM_G}
\end{equation}
and 
\begin{equation}
\begin{aligned}
&\text{P}_{{\rm asy}|T',T,M}
=\\&
\sum_{z=0}^{1}
\text{P}_\text{E}^z(1-\text{P}_\text{E})^{1-z}
\sum_{K_{\rm pre}=0}^{M-1}
\binom{M-1}{K_{\rm pre}}
\text{P}_\text{E}^{K_{\rm pre}}
(1-\text{P}_\text{E})^{M-1-K_{\rm pre}}
\\
&\quad \times
\sum_{r_{\rm pre}=0}^{\min(K_{\rm pre},N-1)}
\Theta(M-1,N,N-1,K_{\rm pre},r_{\rm pre})
\\
&\quad \times
\sum_{K=0}^{T'-M}
\binom{T'-M}{K}
\text{P}_\text{E}^K(1-\text{P}_\text{E})^{T'-M-K}
\\
&\quad \times
\sum_{r'=r''}^{N-1}
\Psi(K,N-r'',r'-r'')
\frac{2^N-2^{r'}}{2^N-1},
\end{aligned}
\label{eq:theorem_Pasy_cond_exact_G}
\end{equation}
where  $r''=r_{\rm pre}+z$. Here, the helper functions are
\begin{equation}
\Psi(m,n,r)=2^{-mn}\Phi(m,n,r),
\label{eq:theorem_Psi_def}
\end{equation}and 
\begin{equation}
\begin{aligned}
&\Theta(m,N,R,K,r'')
\\&=
\frac{
\Phi(K,N,r'')\,
2^{(m-K)r''}\,
\Phi(m-K,N-r'',R-r'')
}{
\Phi(m,N,R)
},
\label{eq:theorem_Theta_final_H}
\end{aligned}
\end{equation}
where
\begin{equation}
\Phi(m,n,r)
=
\frac{
\prod_{i=0}^{r-1}(2^m-2^i)(2^n-2^i)
}{
\prod_{i=0}^{r-1}(2^r-2^i)
}.
\label{eq:theorem_Phi_def_H}
\end{equation}

\end{Theorem}

\begin{proof}The proof is provided in Appendix \ref{appendix:P_asy}.
\end{proof}

Since each transmission block is statistically identical and independent, the first asynchronization incidence follows a geometric process. Therefore, the expected time to the first asynchronization incidence is equivalent to the expected interval between two consecutive incidences.

\begin{corollary}\label{Theorem:T_asy}
The expected time to the first asynchronization incidence, equivalently the expected interval between two consecutive incidences, is
\begin{equation}
\begin{aligned}
\tau_{\text{asy}}=\dfrac{\tau_q}{\mathrm{P}_{\text{asy}}}\big(\dfrac{N+1.6}{\mathrm{P}_\text{B}}+\mathbb{E}[T_{\text{ack}}]\big)
\end{aligned}
\label{eq:T_asy}
\end{equation}
where $\tau_q$ denotes the packet transmission time and $\mathrm{P}_{\text{asy}}$ is described in Theorem  \ref{Lemma:P_async}.
\end{corollary}

Once an asynchronization incidence occurs, Eve must eventually resynchronize by identifying the dummy packets of a subsequent full-rank block. Since these packets are indistinguishable from the original packets, the required brute-force search gives rise to the following computational complexity.

\begin{corollary}\label{Theorem:Computational_power}
Let $\tau_\text{cmp}$
denote the expected computational time required to continuously maintain synchronization during one second of communication through repeated brute-force resynchronization. It is given by
{\begin{equation}
\begin{aligned}
\tau_{\mathrm{cmp}}=\binom{N}{D}\dfrac{ D{\mathrm{P}}_{\mathrm{B}}\mathrm{P}_{\mathrm{asy}}C_\text{E}^{-1}}{\tau_q\times\big(N+1.6+\mathrm{P}_{\mathrm{B}}{\mathbb{E}}\{T_{\text{ack}}\}\big)},
\label{eq:brute_force_computational_power}
\end{aligned}
\end{equation}}where $C_\text{E}$ denotes Eve's computational capacity, measured as the number of packet combinations evaluated per second.
\end{corollary}

Corollary~\ref{Theorem:Computational_power} shows that the computational cost of each resynchronization attempt grows combinatorially with the block parameters. In particular, increasing either $N$ or $D$ rapidly increases the computational burden imposed on Eve. In contrast, the channel conditions and ACK delay affect only the frequency of resynchronization attempts through $\text{P}_{\rm asy}$, rather than the computational cost of each individual attempt.

\section{Performance Metrics and Design Optimization}
The analytical framework enables optimization of block parameters, transmit power, and intentional interference to balance communication performance and computational security.

\subsection{Channel Model}
To formulate the optimization problems, we first introduce the physical-layer model used to evaluate the packet success probabilities of Bob and Eve. We consider a PtP Rayleigh fading channel in which Alice transmits with total power $P$. A fraction $\alpha P$ is allocated to the information signal, while the remaining $(1-\alpha)P$ is allocated to intentional interference. The resulting signal-to-interference-plus-noise ratios (SINRs) are given by
\begin{equation}
\begin{aligned}
&\gamma_{\rm B}
=
\frac{\alpha P |h_{\rm B}|^2}
{(1-\alpha)P |g_{\rm B}|^2+\sigma_{\rm B}^2}, 
\\&\gamma_{\rm E}
=
\frac{\alpha P |h_{\rm E}|^2}
{(1-\alpha)P |g_{\rm E}|^2+\sigma_{\rm E}^2},
\end{aligned}\label{eq:SINRs}
\end{equation}
respectively. Here, $h_{\rm B}\sim\mathcal{CN}(0,\Omega_{\rm B})$ and
$h_{\rm E}\sim\mathcal{CN}(0,\Omega_{\rm E})$ denote the
data-signal channels, 
$g_{\rm B}\sim\mathcal{CN}(0,\Omega_{{\rm B},{\rm AN}})$ and
$g_{\rm E}\sim\mathcal{CN}(0,\Omega_{{\rm E},{\rm AN}})$ denote the intentional-interference channels toward Bob and Eve, respectively, and $\sigma_\text{B}^2$ and $\sigma_\text{E}^2$ denote the corresponding thermal noise powers.

The intentional interference is transmitted from a separate co-located antenna, whose channel is assumed to be statistically independent of the data channel. Since the interference is unknown to both receivers, it is modeled as additional noise. Under Rayleigh fading and assuming capacity-achieving coding with sufficiently long packets, the packet success probabilities are approximated by the outage probabilities
\begin{equation}
\begin{aligned}
\text{P}_{\rm B}
&=
\mathcal{P}\!\left(
\gamma_{\rm B}>\nu(R_t)
\right)
=
\frac{
\exp\!\left(
-\dfrac{\nu(R_t)\sigma_{\rm B}^2}
{\alpha P\Omega_{\rm B}}
\right)
}
{
1+
\dfrac{\nu(R_t)(1-\alpha)\Omega_{{\rm B},{\rm AN}}}
{\alpha\Omega_{\rm B}}
},
\\[1ex]
\text{P}_{\rm E}
&=
\mathcal{P}\!\left(
\gamma_{\rm E}>\nu(R_t)
\right)
=
\frac{
\exp\!\left(
-\dfrac{\nu(R_t)\sigma_{\rm E}^2}
{\alpha P\Omega_{\rm E}}
\right)
}
{
1+
\dfrac{\nu(R_t)(1-\alpha)\Omega_{{\rm E},{\rm AN}}}
{\alpha\Omega_{\rm E}}
},
\end{aligned}
\label{eq:PtP_P_loss}
\end{equation}
   where $\nu(R_t)=2^{R_t}-1$, $R_t=R_c\log _2M_\text{mod}$ is the target rate,  $M_\text{mod}$ is the modulation order, and $R_c\in [0,1]$ is the code rate. 

Under the co-located antenna assumption, the receivers experience approximately the same large-scale channel gains, i.e., $\Omega_{{\rm B},{\rm AN}}\approx\Omega_{\rm B}$ and $\Omega_{{\rm E},{\rm AN}}\approx\Omega_{\rm E}$. So, \eqref{eq:PtP_P_loss} simplifies to
\begin{equation}
\begin{aligned}
\text{P}_{\rm B}
&
=
\alpha\exp\!\left(
-\dfrac{\nu(R_t)\sigma_{\rm B}^2}
{\alpha P\Omega_{\rm B}}
\right)\Big{/}\big(\alpha+
\nu(R_t)(1-\alpha)\big),
\\[1ex]
\text{P}_{\rm E}
&
=
\alpha\exp\!\left(
-\dfrac{\nu(R_t)\sigma_{\rm E}^2}
{\alpha P\Omega_{\rm E}}
\right)\Big{/}\big(\alpha+
\nu(R_t)(1-\alpha)\big),
\end{aligned}
\label{eq:PtP_P_succ}
\end{equation}implying that intentional interference serves as an enhancement by increasing the probability of asynchronization incidence, rather than acting as the primary security mechanism.

\subsection{Performance Metrics}

Using the analytical framework developed in Section IV, the following performance metrics are employed throughout the optimization problems. Using the PDF of $T$ derived in Theorem~\ref{Lemma_Bob_Ack_time_pdf}, the packet latency is defined as
\begin{equation}
    \ell=\tau_q\times{\mathbb{E}\{T\}},
\end{equation}
  which is equivalent to the expected time that the decoding matrix of Bob achieves full rank. Further, the throughput is defined as the average number of original packets successfully delivered per second, described as
\begin{equation}
    \eta=\dfrac{N-D}{\tau_q\times\mathbb{E}\{T'\}}=\dfrac{N-D}{\tau_q\times(\mathbb{E}\{T\}+\mathbb{E}\{T_{\text{ack}}\})},
\end{equation}
where the PDF of $T'$ is defined in \eqref{eq:T_prime_pdf}. Besides communication performance, the optimization also constrains the amount of uncertainty retained by Eve after successful decoding. Accordingly, the interpretation uncertainty is defined as
\begin{equation}
     \epsilon= \dfrac{D}{N}.
\end{equation}A minimum uncertainty requirement ensures that successfully decoded blocks remain insufficient for reliable interpretation without the synchronization state.

\subsection{Transmit Power Optimization}

Transmit power directly affects the packet success probabilities of both Bob and Eve, and therefore, the probability of eavesdropper asynchronization. Specifically, varying the transmit power changes $\mathrm{P}_{\rm B}$ and $\mathrm{P}_{\rm E}$, thereby affecting $\mathrm{P}_{\rm asy}$ and the expected time to the first asynchronization incidence, $\tau_{\rm asy}$. Since reducing $\tau_{\rm asy}$ increases the frequency with which Eve must perform computationally expensive resynchronization, it is selected as the optimization objective. We first consider the baseline system without intentional interference ($\alpha=1$) to isolate the effect of transmit-power optimization. Let
\[
\boldsymbol{\Upsilon}_1
=
(\alpha=1,N,D,\Omega_{\rm B},\Omega_{\rm E},
\sigma_{\rm B},\sigma_{\rm E},
R_t,\text{f}_{\rm ACK}(t))
\]
collect the fixed system parameters. The transmit power is then optimized by solving
\begin{subequations}\label{eq:P1}
\begin{align}
\mathbb{P}_1:\quad
\underset{P}{\text{minimize}}\quad
& \tau_{\mathrm{asy}}(P,\boldsymbol{\Upsilon}_1)
\\
\text{s.t.}\quad
& \eta(P,\boldsymbol{\Upsilon}_1) \geq \eta_{\min},
\\
& \ell(P,\boldsymbol{\Upsilon}_1) \leq \ell_{\max},
\\
& \tau_{\mathrm{cmp}}(P,\boldsymbol{\Upsilon}_1) \geq \tau_{\min},
\\
& 0<P \leq P_{\max}.
\end{align}
\end{subequations}
The objective function is given in \eqref{eq:T_asy}, while constraints \eqref{eq:P1}b--\eqref{eq:P1}d enforce the throughput, latency, and computational-security requirements, respectively.

Problem $\mathbb{P}_1$ provides a useful performance benchmark when Alice has knowledge of Eve's channel gain, or at least a conservative upper bound on it. Although such information is generally unavailable for passive eavesdroppers, the formulation establishes the best performance achievable through transmit-power adaptation alone. Nevertheless, when Eve's channel is substantially stronger than Bob's ($\Omega_{\rm E}\gg\Omega_{\rm B}$), transmit-power optimization cannot fundamentally compensate for Eve's channel advantage. Therefore, although it provides useful performance improvements, it may still yield unacceptably large values of $\tau_{\rm asy}$. This limitation motivates the intentional-interference optimization developed in the next subsection, where Eve's packet success probability is reduced independently of her channel strength.

\begin{figure*}
    \centering
    \vspace{0cm}\includegraphics[width=1\linewidth]{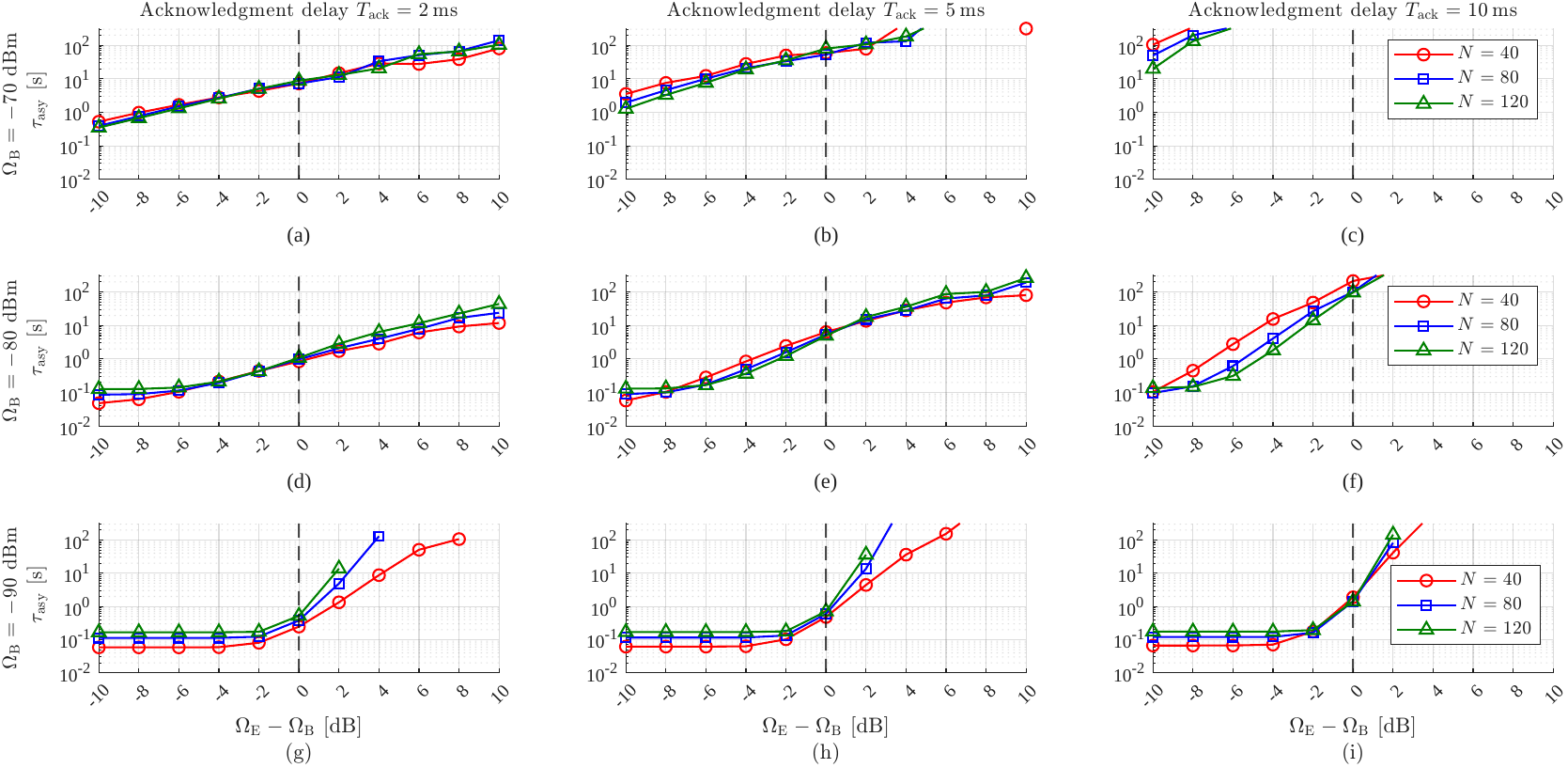}
    \vspace{-0.8cm}
    \caption{Mean time to the first asynchronization incidence, equivalently the mean interval between two consecutive asynchronization incidences, $\tau_{\mathrm{asy}}$, for transmit power $P=1$\,Watt versus the channel gain difference between Eve and Bob, $\Omega_{\mathrm{E}}-\Omega_{\mathrm{B}}$, for Bob's mean channel gains $\Omega_{\mathrm{B}}=-70$, $-80$, and $-90$\,dBm, ACK delays $T_{\rm ack}=2$, $5$, and $10$\,ms, and block sizes $N=40$, $80$, and $120$.}
    \label{fig:async_time}
\end{figure*}

\subsection{Intentional-Interference Optimization}

Although transmit-power optimization improves the probability of eavesdropper asynchronization, it cannot fundamentally compensate for an arbitrarily strong eavesdropper. In contrast, intentional interference introduces a controllable impairment that limits Eve's packet success probability even when her received signal power is very large. Hence, it reduces the expected time to the first asynchronization incidence, $\tau_{\rm asy}$. The optimization variable $\alpha$ determines the power split between the information signal and the intentional interference. Accordingly, $\alpha$ is optimized by solving
\begin{subequations}\label{eq:P2_artificial_noise}
\begin{align}
\mathbb{P}_2:\quad
\underset{\alpha}{\text{minimize}}\quad
& \tau_{\mathrm{asy}}(\alpha,\boldsymbol{\Upsilon}_2)
\\
\text{s.t.}\quad
& \eta(\alpha,\boldsymbol{\Upsilon}_2) \geq \eta_{\min},
\\
& \ell(\alpha,\boldsymbol{\Upsilon}_2) \leq \ell_{\max},
\\
& \tau_{\mathrm{cmp}}(\alpha,\boldsymbol{\Upsilon}_2) \geq \tau_{\min},
\\
& 0<\alpha\leq1,
\end{align}
\end{subequations}
where
\[
\boldsymbol{\Upsilon}_2
=
(P=P_{\max},N,D,\Omega_{\rm B},
\Omega_{\rm E}=\infty,
\sigma_{\rm B},\sigma_{\rm E},
R_t,f_{\rm ACK}(t))
\]
collects the fixed system parameters.

To obtain a design that remains effective regardless of Eve's channel strength, we consider the worst-case limit $\Omega_{\rm E}\rightarrow\infty$. Under this assumption, \eqref{eq:PtP_P_succ} reduces to
\begin{equation}
\begin{aligned}
\text{P}_{\rm E}
=
\frac{\alpha}
{\alpha+\nu(R_t)(1-\alpha)},
\end{aligned}
\label{eq:PtP_P_succ_Eve_upperbound}
\end{equation}
which depends solely on the power-allocation factor $\alpha$. Hence, unlike transmit-power optimization, intentional interference fundamentally bounds Eve's packet success probability even under arbitrarily strong channel conditions. Since increasing the transmit power no longer affects Eve's packet success probability, the optimal strategy is to set $P=P_{\max}$ and optimize only $\alpha$. The resulting optimum is denoted by $\alpha^*$, and the corresponding objective value by $\tau_{\rm asy}^{\rm upp}$, which represents an upper bound on the expected time to the first asynchronization incidence because $\Omega_{\rm E}=\infty$ corresponds to the most favorable propagation condition for Eve.




\begin{figure*}
    \centering
    \vspace{0cm}\includegraphics[width=1\linewidth]{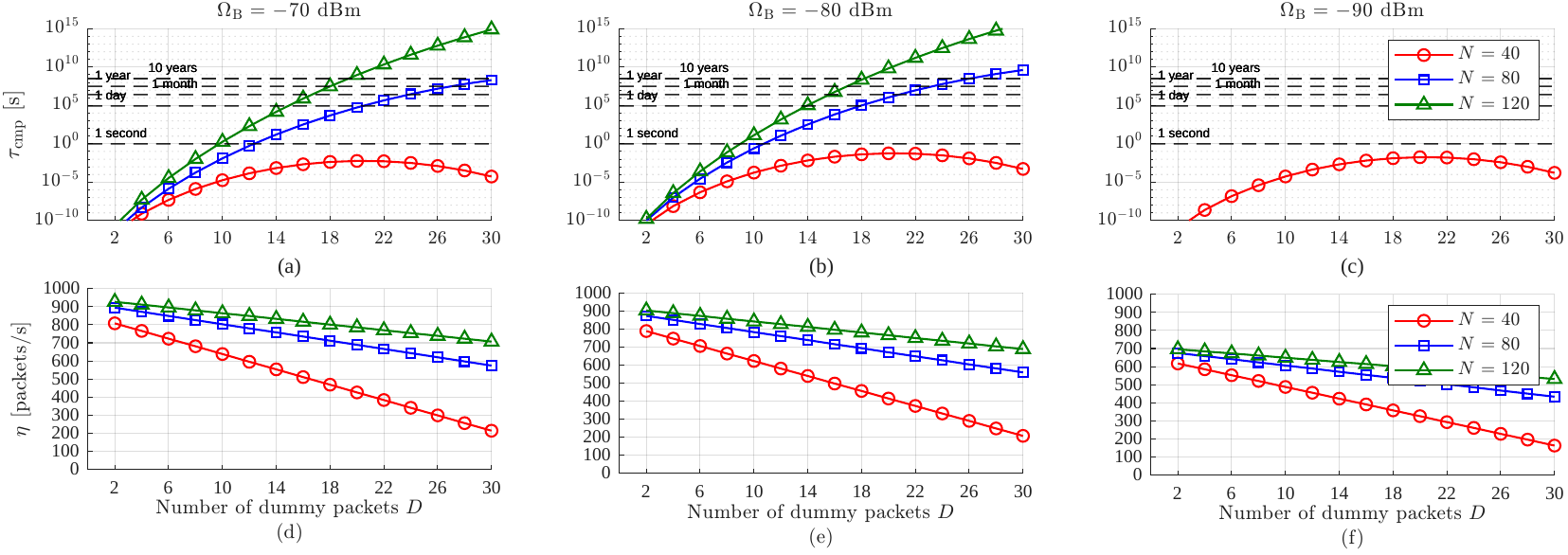}
    \vspace{-0.8cm}
    \caption{Impact of the number of dummy packets $D$ on computational security and throughput. The top row shows the expected computational time required by Eve to maintain synchronization through repeated brute-force resynchronization, $\tau_{\rm cmp}$ for one second of transmission with the proposed state-chained scheme, and the bottom row shows the resulting throughput $\eta$.}
    \vspace{-0.2cm}
    \label{fig:complexity_throughput}
\end{figure*}

\section{Simulation results}\label{Section:simulation_results}

This section validates the analytical framework and evaluates how temporal dependency behaves under increasingly adverse eavesdropping conditions. For the simulations, we consider a PtP link in the presence of a passive eavesdropper. The total available transmit power at Alice is $P_{\rm max}=1$\,Watt and the target rate is set as $R_t=2$\,bpcu. Each packet has a payload of $n=1$\,kilobyte and transmission time $\tau_q=1$\,ms, while the receiver noise powers are $\sigma_{\rm B}^2=\sigma_{\rm E}^2=-100$\,dBm. Unless otherwise stated, the results are presented as a function of the channel gain difference $\Omega_{\rm E}-\Omega_{\rm B}$, which provides a more informative comparison than the absolute channel gain of Eve.

Fig.~\ref{fig:async_time} evaluates the baseline temporal-dependency security performance before optimization, demonstrating how quickly Eve becomes asynchronized under different channel conditions. Specifically, it shows the mean time to the first asynchronization incidence, $\tau_{\rm asy}$, obtained from Monte Carlo simulations with $10^4$ independent realizations, as a function of the channel gain difference $\Omega_{\rm E}-\Omega_{\rm B}$ for $N=40$, $80$, and $120$. The region $\Omega_{\rm E}>\Omega_{\rm B}$ corresponds to a stronger eavesdropper, where the secrecy capacity becomes zero (or negative in theory). As expected, $\tau_{\rm asy}$ increases with Eve's channel gain because the probability of asynchronization, $\mathrm{P}_{\rm asy}$, decreases. Likewise, increasing the ACK delay from $2$ to $10$\,ms substantially degrades performance for all block sizes, since the longer observation window allows Eve to collect more innovative packets before transmission terminates. Nevertheless, for $T_{\rm ack}=2$\,ms, $\tau_{\rm asy}$ remains below approximately $10$--$30$\,s even when Eve's channel is $10$\,dB stronger than Bob's. Furthermore, larger block sizes are more resilient to long ACK delays because Eve requires more innovative packets to achieve full rank, whereas smaller block sizes become preferable under poor channel conditions. These results motivate the formulation of the transmit-power optimization problem $\mathbb{P}_1$
 and the intentional-interference optimization problem $\mathbb{P}_2$
 to reduce the expected time to the first asynchronization incidence.

Fig.~\ref{fig:complexity_throughput} evaluates the computational security provided by the proposed framework by showing the expected computational time required for Eve to maintain synchronization over one second of communication through repeated brute-force resynchronization. To represent an extremely powerful adversary, Eve is assumed to employ 100 NVIDIA B200 GPUs operating in parallel, providing a total computational capacity of $C_{\rm E}=10^{12}$ packet combinations per second, while her average channel gain is assumed to be 10\,dB higher than Bob's. The computational burden is highly sensitive to the block parameters $N$ and $D$. For $D\le8$, the required computation time remains below one second for all block sizes, whereas it increases combinatorially with $D$. Although the computation time remains below one second for $N=40$, it exceeds ten years at approximately $D=30$ for $N=80$ and $D=20$ for $N=120$. Furthermore, for $N=120$, increasing $D$ from $14$ to $18$ increases the expected resynchronization time from approximately one day to several years. These results demonstrate that, once an asynchronization incidence occurs, maintaining synchronization becomes computationally infeasible even for an exceptionally powerful adversary, without relying on cryptographic assumptions.

The bottom row of Fig.~\ref{fig:complexity_throughput} shows the corresponding throughput. As expected, increasing the number of dummy packets reduces throughput because a larger fraction of each block is devoted to synchronization-state information. Conversely, for a fixed $D$, larger block sizes improve throughput by reducing the relative overhead introduced by the dummy packets. This trade-off motivates optimizing $N$ and $D$ rather than simply maximizing computational complexity.

\begin{figure*}
    \centering
    \vspace{0cm}\includegraphics[width=1\linewidth]{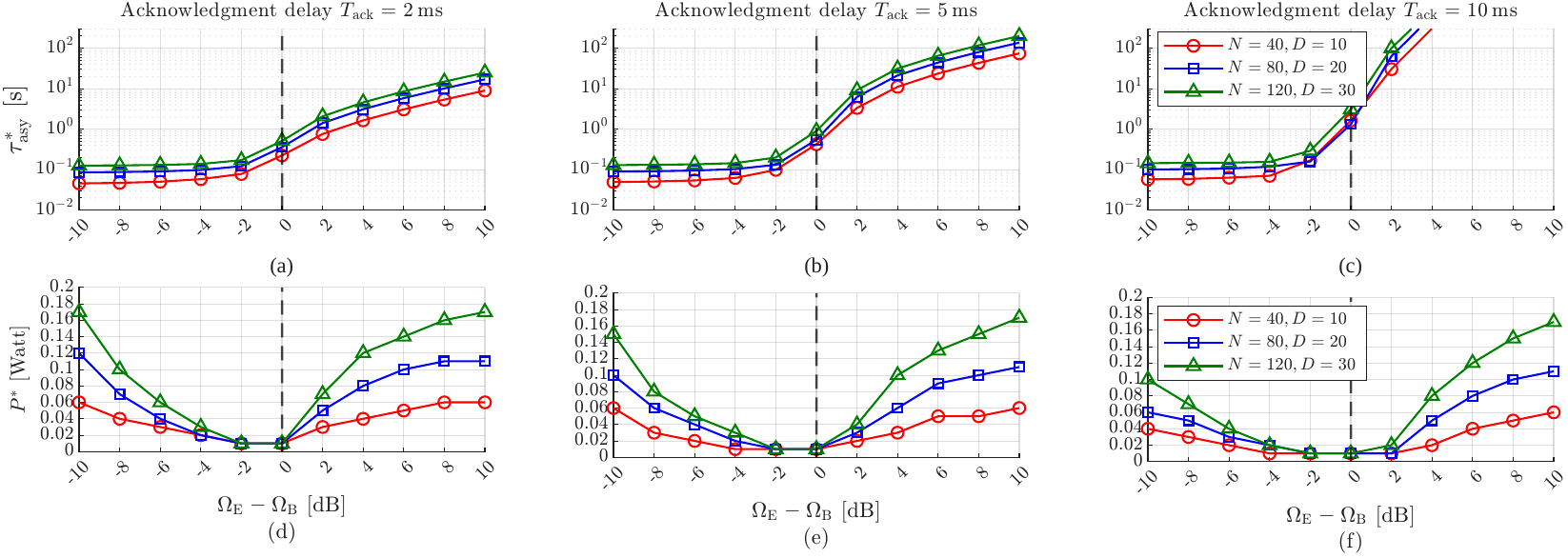}
    \vspace{-0.8cm}
    \caption{Optimal solution of $\mathbb{P}_1$. Optimal asynchronization time $\tau
_{\rm asy}^{*}$ (top row) and optimal transmit power $P^{*}$ (bottom row) versus the channel gain difference $\Omega_{\rm{E}}-\Omega_{\rm{B}}$, by setting $\Omega_{\rm{B}}=-70$\,dBm and $P_{\rm max}=1$\,Watt.}
\vspace{-0.2cm}
    \label{fig:OP1}
\end{figure*}

Fig.~\ref{fig:OP1} demonstrates the benefit of transmit-power optimization in reducing the expected time to the first asynchronization incidence. The figure presents the optimal solution of $\mathbb{P}_1$ with $P_{\max}=1$\,Watt and a minimum throughput requirement of $\eta_{\min}=500$\,packets/s. The top row shows the optimal expected time to the first asynchronization incidence, $\tau_{\rm asy}^{*}$, while the bottom row shows the corresponding optimal transmit power, $P^{*}$, for $\Omega_{\rm B}=-70$\,dBm and ACK delays of $T_{\rm ack}=2$, $5$, and $10$\,ms. As expected, $\tau_{\rm asy}^{*}$ increases as Eve's channel becomes stronger. Nevertheless, it remains below approximately $10$--$20$\,s for $T_{\rm ack}=2$\,ms even when Eve's channel is $10$\,dB stronger than Bob's, and below $1$\,s for $T_{\rm ack}\le5$\,ms when Bob and Eve experience similar channel conditions. The optimal transmit power is approximately $0.01$\,Watt when the channel conditions are comparable, since a lower transmit power slightly delays Bob's decoding while increasing the likelihood that Eve remains rank deficient, thereby reducing $\tau_{\rm asy}$. As either Bob's or Eve's channel becomes significantly stronger, the optimal transmit power increases to satisfy the throughput constraint while minimizing $\tau_{\rm asy}$. Compared with the baseline results in Fig.~\ref{fig:async_time}, transmit-power optimization substantially reduces $\tau_{\rm asy}$, particularly when Bob and Eve experience similar channel conditions. However, its performance remains fundamentally limited against arbitrarily strong eavesdroppers, motivating the intentional-interference optimization developed in $\mathbb{P}_2$.

Fig.~\ref{fig:OP2} presents the optimal solution of $\mathbb{P}_2$, demonstrating that intentional interference fundamentally removes the dependence of communication security on channel advantage. The top row shows the resulting upper bound on the expected time to the first asynchronization incidence, $\tau_{\rm asy}^{\rm upp}$, while the bottom row shows the corresponding optimal intentional-interference coefficient, $\alpha^{*}$, as a function of the minimum throughput requirement, $\eta_{\min}$. The results assume $\Omega_{\rm B}=-70$\,dBm, $\Omega_{\rm E}=\infty$, $P=P_{\max}=1$\,Watt, and $D=N/4$. As expected, increasing the throughput requirement increases $\alpha^{*}$, allocating more power to the information signal and less to intentional interference. Therefore, $\tau_{\rm asy}^{\rm upp}$ increases with $\eta_{\min}$. Nevertheless, the proposed optimization maintains sub-second performance for all considered throughput requirements when $T_{\rm ack}=2$\,ms, and for throughput requirements below approximately $600$\,packets/s when $T_{\rm ack}=5$\,ms. These results demonstrate that sub-second eavesdropper asynchronization remains achievable even under the worst-case assumption of $\Omega_{\rm E}=\infty$. However, for $T_{\rm ack}=10$\,ms, the longer observation window outweighs the benefit of intentional interference, making packet interleaving necessary to further shorten Eve's observation interval.
\begin{figure*}
    \centering
    \vspace{0cm}\includegraphics[width=1\linewidth]{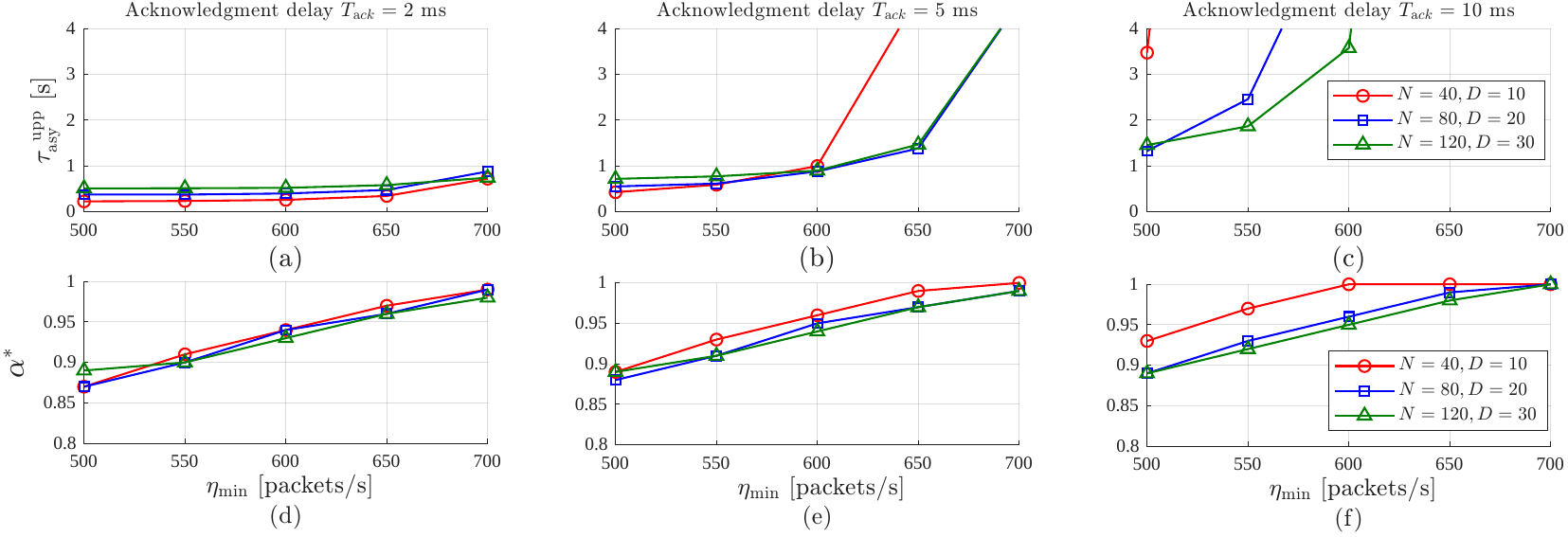}
    \vspace{-0.8cm}
    \caption{Optimal solution of $\mathbb{P}_2$; Upper bound to the mean time of the first asynchronization incidence $\tau
_{\rm asy}^{\rm upp}$ (top row) and respective optimal intentional interference coefficient $\alpha^{*}$ (bottom row) versus the minimum required throughput $\eta_{\rm min}$ for three ACK delay values $T_{\rm ack}=2,5$ and $10$\,ms, by setting $\Omega_{\rm{B}}=-70$\,dBm, $\Omega_{\rm{E}}=\infty$, $D=N/4$, and $P=P_{\rm max}=1$\,Watt.} 
 \vspace{-0.2cm}
    \label{fig:OP2}
   
\end{figure*}

\section{Conclusion}
This paper introduced a new communication-security paradigm in which temporal dependency serves as the underlying security resource, allowing the successful interpretation of future communications to depend on correctly interpreting previous transmissions. As one realization, we developed a state-chained RLNC framework together with analytical models for eavesdropper asynchronization, resynchronization complexity, and the optimization of transmit power and intentional interference. Numerical results showed that transmit-power optimization substantially reduces the expected time to the first eavesdropper asynchronization incidence, while intentional-interference optimization achieves sub-second expected asynchronization even under a worst-case adversarial model with no CSI, complete protocol and cryptographic knowledge, and an arbitrarily stronger eavesdropper. Meanwhile, resynchronization imposes computational costs of several years on the eavesdropper. These results demonstrate that communication security can be achieved without relying on secret-key confidentiality or a communication advantage, making temporal-dependency security a promising paradigm for future infrastructure-less and highly dynamic wireless networks.

\appendices
\section{Proof of Theorem \ref{Lemma_Bob_Ack_time_pdf}} \label{Appendix_A}
The probability that Bob becomes full rank  exactly at packet time $t=T$ is described as the joint probability of the rank be $N-1$ at the previous time slot $T-1$ and a rank increment at time slot $T$ as follows
\begin{equation}
\begin{aligned}
&\mathcal{P}\Big(\text{r}({\bf{B}}_T)=N,\text{r}({\bf{B}}_{T-1})=N-1\Big)  \\& 
=\mathcal{P}\Big(\text{r}({\bf{B}}_{T-1})=N-1\Big)\mathcal{P}\Big(\text{r}({\bf{B}}_T)=N|\text{r}({\bf{B}}_{T-1})=N-1\Big)  
\label{eq:Bob_Ack_time_pdf}
\end{aligned}
\end{equation}
where the second term in \eqref{eq:Bob_Ack_time_pdf} correspond to the probability that Bob successfully receive the  packet at time slot $t=T$ and that a random vector of size $N$ be out of the row span of a random full rank matrix in $\mathbb{F}^{N-1\times N}$, and derived as 
\begin{equation}
\begin{aligned}
\mathcal{P}\Big(\text{r}({\bf{B}}_T)=N|\text{r}({\bf{B}}_{T-1})=N-1\Big)=\dfrac{2^{N-1}}{2^N-1}\mathrm{P}_\text{B}.
\label{eq:Bob_Ack_fullrank_conditional}
\end{aligned}
\end{equation}
Moreover, the first term in \eqref{eq:Bob_Ack_time_pdf} can be expanded based on the number of rows in ${\bf{B}}_{T-1}$ as
\begin{equation}
\begin{aligned}
&\mathcal{P}\big(\text{r}({\bf{B}}_{T-1})=N-1\big)\\&
=\sum_{M=N}^{T}\mathcal{P}\big(\text{r}({{\bf{B}}_{T-1}})=N-1|\text{m}({{\bf{B}}_{T-1}})=M-1)\\&~~\times
\mathcal{P}\Big(\text{m}({{\bf{B}}_{T-1}})=M-1\Big)
\end{aligned}\label{eq:rank_bob_N-1}
\end{equation}
where the second term correspond to the probability of $M-1$ successful packet reception by Bob within $T-1$ packet transmission, described as
\begin{equation}
\begin{aligned}
\mathcal{P}\Big(\text{m}({{\bf{B}}_{T-1}})=M-1\Big) =\binom{T-1}{M-1} \mathrm{P}_\text{B}^{M-1} (1-\mathrm{P}_\text{B})^{T-M}
\label{eq:combination_M-1_from_T-1}
\end{aligned}
\end{equation}
Also, the first term in \eqref{eq:rank_bob_N-1} correspond to the the probability of a random matrix in $\mathbb{F}^{M-1\times N}$ with $M\geq N$ has the rank of $N-1$, and derived as
\begin{equation}
\begin{aligned}
&\mathcal{P}\big(\text{r}({{\bf{B}}_{T-1}})=N-1|\text{m}({{\bf{B}}_{T-1}})=M-1)\\&
=\dfrac{\prod_{i=0}^{N-2} (2^{M-1}-2^i)\prod_{i=0}^{N-2}(2^N-2^i)}{2^{(M-1)N}\prod_{i=0}^{N-2}(2^{N-1}-2^i)}
\label{eq:P_rank_r_in_matrix_m_by_N_appendix_A}
\end{aligned}
\end{equation}
whose proof is in Appendix \ref{Appendix:rank_r_with_m_rows}. So, substituting $m$ with $M-1$ and $r$ with $N-1$ in \eqref{eq:P_rank_r_in_matrix_m_by_N},  \eqref{eq:P_rank_r_in_matrix_m_by_N_appendix_A} is derived. Finally, by replacing \eqref{eq:combination_M-1_from_T-1} and \eqref{eq:P_rank_r_in_matrix_m_by_N_appendix_A} into  \eqref{eq:rank_bob_N-1} and then replacing \eqref{eq:Bob_Ack_fullrank_conditional} and \eqref{eq:rank_bob_N-1} into \eqref{eq:Bob_Ack_time_pdf}, Theorem \ref{Lemma_Bob_Ack_time_pdf} is concluded.

\section{Probability of a random matrix in $\mathbb{F}_2^{m\times N}$ be rank $r$}\label{Appendix:rank_r_with_m_rows} 

To determine the probability that the rank of a random matrix ${\bf{B}}\in \mathbb{F}_2^{m*N}$ be $r$, we determine the the number of possible matrices ${\bf{B}}\in \mathbb{F}_2^{m*N}$ with rank $r$ and divide it by all the possible matrices ${\bf{B}}\in \mathbb{F}_2^{m*N}$. To this end, we firstly consider that every rank $r$ matrix can be decomposed as 
\begin{equation}
\begin{aligned}
\bf{B}=\bf{CD}
\label{eq:matrix_factorization}
\end{aligned}
\end{equation}
where ${\bf{C}}\in \mathbb{F}_2^{m*r}$ has full column rank $r$ and ${\bf{D}}\in \mathbb{F}_2^{r*N}$ has full row rank $r$. To prove the existence of such decomposition, we consider ${\bf{D}}$ as the stack of basis vectors of the row space of  ${\bf{B}}$. Since r(${\bf{B}})=r$, the number of the basis vectors of ${\bf{D}}$ is also $r$ and they are linearly independent, and therefore, r(${\bf{D}})=r$. Since each row of $\bf{B}$ is a linear combination of the rows of $\bf{D}$, there exists a matrix $\bf{C}$ such that ${\bf{B}}=\bf{C}\bf{D}$. Then, since r$({\bf{B}})=$ r$({\bf{D}})=r$, we must have r$({\bf{C}})=r$.

To count the number of possible ${\bf{B}}\in \mathbb{F}_2^{m*N}$ with  r(${\bf{B}}$)=$r$, we first count the number full column rank matrices  ${\bf{C}}\in\mathbb{F}_2^{m\times r}$. The $r$ columns of ${\bf{C}}$ must
be linearly independent. Thus, the first column can be any nonzero vector
in $\mathbb{F}_2^{m}$.  The second column must lie outside the span of the first, and so
on. Hence, the number of such matrices is
\begin{equation}
  |\{ {\bf{C}}\in \mathbb{F}_2^{m\times r} : {\textbf{r}}({\bf{C}})=r \}| = \prod_{i=0}^{r-1}(2^{m}-2^i).  \label{eq:number_of_full_rank_C}
\end{equation}
 By the same argument,
the number of  matrices ${\bf{C}}\in \mathbb{F}_2^{r\times N}$ 
 with rank $r$  is
\begin{equation}
     |\{ {\bf{D}}\in \mathbb{F}_2^{r\times N} : {\text{r}}({\bf{D}})=r \}|=\prod_{i=0}^{r-1}(2^{N}-2^i).
\label{eq:number_of_full_rank_D}
\end{equation}
However, the decomposition \eqref{eq:matrix_factorization} is not unique. This is because for any invertible matrix ${\bf{V}}\in\mathbb{F}_2^{r\times r}$, we can rewrite \eqref{eq:matrix_factorization} as  
\begin{equation}
\begin{aligned}
\bf{B}=\bf{CD}=(\bf{CV})(\bf{V}^{-1}{D})
\label{eq:matrix_factorization_duplications}
\end{aligned}
\end{equation}
implying that different pairs  of $\bf{C}$ and $\bf{D}$ in \eqref{eq:matrix_factorization}  can generate the same matrix  $\bf{B}$. Therefore, the number of pairs of $\bf{C}$ and $\bf{D}$ that generates the same matrix $\bf{B}$ is equivalent to the number of all possible full rank matrices in $\mathbb{F}_2^{r\times r}$, given by
\begin{equation}
     |\{ {\bf{V}}\in \mathbb{F}_2^{r\times r} : {\text{r}}({\bf{D}})=r \}|=\prod_{i=0}^{r-1}(2^{r}-2^i).
\label{eq:number_of_full_rank_Q}
\end{equation}
Finally, from \eqref{eq:number_of_full_rank_C}, \eqref{eq:number_of_full_rank_D}, and \eqref{eq:number_of_full_rank_Q}, the probability that a random
matrix ${\bf{B}}\in \mathbb{F}_2^{m\times N}$ be rank $r$ is described as
\begin{equation}
\begin{aligned}
&\mathcal{P}\big(\text{r}({\bf{B}})=r|\text{m}({{\bf{B}}})=m)\\&
=2^{-mN}\dfrac{\prod_{i=0}^{r-1} (2^{m}-2^i)\prod_{i=0}^{r-1}(2^N-2^i)}{\prod_{i=0}^{r-1}(2^{r}-2^i)}
\label{eq:P_rank_r_in_matrix_m_by_N}
\end{aligned}
\end{equation}
where the multiplier refers to the total number of possible matrices ${\bf{B}}\in \mathbb{F}_2^{m\times N}$.

\section{Helper function to determine the probability of asynchronization incidence}\label{appendix:Helper function}

Lets ${\bf{A}}_{[K]}\in \mathbb{F}^{K\times N}$ denotes the sub-matrix formed by a random selection of $K$ rows of a matrix ${\bf{A}}\in \mathbb{F}^{m\times N}$ with rank $R$. We define the probability that ${\bf{A}}_{[K]}$  has rank $r''$ as
\begin{equation}
\begin{aligned}
&\mathcal{P}\!\left(
{\rm{ r}}({\bf{A}}_{[K]})=r''
\mid 
{\rm r}({\bf A})=R, {\bf{A}}_{[K]} \subseteq_K{\bf{A}} \in \mathbb{F}_2^{m\times N}
\right)\\&~~~~~~~~~~~~~~~~~~=:\Theta(m,N,R,K,r'').
\label{eq:Theta_def_H}
\end{aligned}
\end{equation}
To derive \eqref{eq:Theta_def_H}, we count the number of
 matrices ${\bf A}\in\mathbb{F}_2^{m\times N}$ with rank $R$ whose selected $K$ rows
have rank $r{''}$ and divide by the total number of matrices ${\bf A}\in\mathbb{F}_2^{m\times N}$  with rank $R$. From Appendix \ref{Appendix:rank_r_with_m_rows}, the total number of matrices in $\mathbb{F}_2^{m\times N}$  with
rank $R$ is given by
\begin{equation}
|\{ {\bf{A}}\in \mathbb{F}_2^{m\times N} : {\textbf{r}}({\bf{A}})=R \}|=:\Phi(m,N,R),
\label{eq:total_rank_R_H}
\end{equation}
where $\Phi(m,n,r)$ is defined in \eqref{eq:theorem_Phi_def_H}.

Next, consider the event that the selected $K$ rows have rank
exactly $r''$. The number of possible choices for the $K$ rows
is
\begin{equation}
|\{ {\bf{A}}_{[K]}\in \mathbb{F}_2^{K\times N} : {\text{r}}({\bf{A}}_{[K]})=r'' \}|=\Phi(K,N,r'').
\label{eq:selected_rows_H}
\end{equation}
Since $\text{r}({\bf{A}}_{[K]})=r''$, the rows of ${\bf{A}}_{[K]}$ span an $r''$-dimensional subspace $\mathcal{U}\subseteq \mathbb{F}_2^N$  with the dimension $\dim({\cal U})=r''$. 
The remaining (non-selected) $m-K$ rows in ${\bf{A}}\in \mathbb{F}^{m\times N}$ must increase the rank from $r''$ to
$R$. Let ${\cal W}$ be a complementary subspace of ${\cal U}$
such that
\begin{equation}
  {\cal U} \oplus {\cal W}=\mathbb{F}_2^N,\qquad \dim({\cal W})=N-r''.
\end{equation}
Then, any remaining  row ${\bf y}\in\mathbb{F}_2^N$ in ${\bf{A}}\in \mathbb{F}^{m\times N}$ can be uniquely
decomposed as
\begin{equation}
{\bf y}={\bf u}\oplus {\bf w}, \qquad {\bf u}\in{\cal U}\subseteq \mathbb{F}_2^N,
\qquad
{\bf w}\in{\cal W}\subseteq \mathbb{F}_2^N.
\end{equation}

 Since ${\cal U}$ is an $r''$-dimensional subspace, there are exactly  $2^{r''}$ selections for $\bf{u}$. Therefore, for each of the
$m-K$ remaining rows, $\bf{u}$ can be
chosen in $2^{r''}$  ways. Then, the total number
of possible choices for the components inside ${\cal U}$ is
\begin{equation}
\begin{aligned}
&|\{ {\bf{U}}\in \mathbb{F}_2^{(m-K)\times N} : {\text{row}}({\bf{U}})\subseteq \mathcal{U},\dim({\cal U})=r'' \}|\\&
~~~~~~~~~~~~~=2^{(m-K){r''}},
\label{eq:inside_U_H}
\end{aligned}
\end{equation}
where each selection of $\bf{u}$ forms a row in the matrix $U$. Choosing a basis for ${\cal W}$ establishes a one-to-one
correspondence between any  vector $\bf{w}$ in ${\cal W}$ and their
coordinate representations in $\mathbb{F}_2^{N-{r''}}$.
Therefore, the components of the remaining $m-K$ rows that
lie in ${\cal W}$ can be represented by a matrix ${\bf{W}}\in \mathbb{F}^{(m-K)\times (N-r'')}$. In order to increase the
overall rank from $r''$ to $R$, this matrix must have rank
$R-r''$. Consequently, the number of such possibilities is
given by
\begin{equation}
\begin{aligned}
&|\{ {\bf{W}}\in \mathbb{F}_2^{(m-K)\times (N-r'')} :  {\text{r}}({\bf{W}})=R-r'' \}|\\&~~~~~~~~~~~~~~~~~=\Phi(m-K,N-r'',R-r'').
\end{aligned}
\end{equation}
Then, the total number of rank-$R$ matrices in $\mathbb{F}^{m\times N}$ whose selected $K$ rows have rank $r''$ is
\begin{equation}
\begin{aligned}
&|\{ \! {\bf{A}} \subseteq \mathbb{F}_2^{m\times N}
: 
{\rm r}({\bf A})=R, {\bf{A}}_K \subseteq_K {\bf{A}}, {\rm{ r}}({\bf{A}}_K)=r'' \}|
\\&
=\Phi(K,N,r'')\,
2^{(m-K)r''}\,
\Phi(m-K,N-r'',R-r'').
\label{eq:numerator_H}
\end{aligned}
\end{equation}Dividing by the total number of rank-$R$ matrices, the expression in \eqref{eq:Theta_def_H} is derived as
\begin{equation}
\begin{aligned}
&\Theta(m,N,R,K,r'')
\\&=
\frac{
\Phi(K,N,r'')\,
2^{(m-K)r''}\,
\Phi(m-K,N-r'',R-r'')
}{
\Phi(m,N,R)
}.
\label{eq:Theta_final_H}
\end{aligned}
\end{equation}which is used in appendix \ref{appendix:P_asy} for determining the probability of asynchronization incidence at Eve.

\section{Probability of asynchronization incidnece}\label{appendix:P_asy}
Asynchronization occurs when the state vector does not belong to the row space of Eve's decoding matrix at the end of a transmission block, i.e., \eqref{eq:async_condition}. The probability of this event is obtained by successively conditioning on Bob's decoding event, the packets shared between Bob and Eve, the packets received only by Eve, and the final rank of Eve's decoding matrix. Accordingly,
\begin{equation}
\begin{aligned}
&\text{P}_{\rm asy}=\mathcal{P}\!\left(
{\bf s}\notin {\rm row}({\bf E}_{T'})
\right)
\\&
=\sum_{T'=N}^{\infty}
\sum_{T=N}^{T'}
\text{f}_{\rm ACK}(T'-T)
\sum_{M=N}^{T}
\text{f}_{\text{B},M}(T,M)
\text{P}_{{\rm asy}|T',T,M},
\end{aligned}
\label{eq:Pasy_exact_G}
\end{equation}
where $\text{P}_{{\rm asy}|T',T,M}$ denotes the conditional asynchronization probability given that Bob first achieves full rank at time slot $T$ after successfully receiving $M$ packets, while Alice terminates transmission at time slot $T'$. The term $\mathrm{f}_{\mathrm{B},M}(T,M)$ is the probability of this event and is given by
\begin{equation}
\begin{aligned}
\text{f}_{\text{B},M}&(T,M)
=
\mathcal{P}\!\left(
{\rm r}({\bf B}_{T-1})=N-1
\mid
\text{m}({\bf B}_{T-1})=M-1
\right)
\\
& \times
\binom{T-1}{M-1}
\text{P}_\text{B}^{M-1}(1-\text{P}_\text{B})^{T-M}
\text{P}_\text{B}
\frac{2^N-2^{N-1}}{2^N-1}.
\end{aligned}
\label{eq:fBM_G}
\end{equation}
Here, the first factor represents the probability that Bob's first $M-1$ successfully received coding vectors have rank $N-1$. Since these vectors are independently generated, we have
\begin{equation}
\begin{aligned}
&\mathcal{P}\!\left(
{\rm r}({\bf B}_{T-1})=N-1
\mid
\text{m}({\bf B}_{T-1})=M-1
\right)
\\&~~~~~~~~~~~~~~~~~=
2^{-(M-1)N}
\Phi(M-1,N,N-1),
\label{eq:rank_B_prev_G}
\end{aligned}
\end{equation}
where  $\Phi(.,.,.)$ is described in \eqref{eq:theorem_Phi_def_H}. 

Conditioned on $(T',T, M)$, the 
asynchronization probability is
\begin{equation}
\begin{aligned}
&\text{P}_{\rm asy}(T',T,M)=\sum_{z=0}^{1}
\mathcal{P}(Z=z)
\\&
~~~~\times \sum_{k_{\rm pre}=0}^{M-1}
\mathcal{P}({\bf{s}} \notin \text{row}({\bf{E}}_{T'}) \mid k_{\text{pre}},z)\mathcal{P}(K_{\text{pre}}=k_{\text{pre}})
\end{aligned}
\label{eq:Pasy_(T',T,M)}
\end{equation}
where $K_{\rm pre}$ denotes the number of packets  that are  received by
Eve among the first $M-1$ packets received by Bob. Thus,
\begin{equation}
\mathcal{P}(K_{\rm pre}=k_{\rm pre})
=
\binom{M-1}{{k_{{\rm pre}}}}
\text{P}_\text{E}^{k_{{\rm pre}}}(1-\text{P}_\text{E})^{M-1-{k_{{\rm pre}}}}.
\label{eq:Kprev_G}
\end{equation}
Also,  $Z\in\{0,1\}$ indicates whether Eve also receives Bob's
final innovative packet at time slot $T$. Hence,
\begin{equation}
\mathcal{P}(Z=z)
=
\text{P}_\text{E}^z(1-\text{P}_\text{E})^{1-z}.
\label{eq:z_G}
\end{equation}

Conditioning further on the rank of Eve's shared decoding matrix, denoted by $R_{\rm pre}$, gives
\begin{equation}
\begin{aligned}
& \mathcal{P}({\bf{s}} \notin \text{row}({\bf{E}}_{T'})\mid k_{\text{pre}},z)= \\&\sum_{r_{\rm pre}=0}^{\substack{\min(k_{\rm pre}\\, N-1)}}
\mathcal{P}({\bf{s}} \notin \text{row}({\bf{E}}_{T'}) \mid r_{\rm pre}, k_{\text{pre}},z)\mathcal{P}(R_{\rm pre}=r_{\rm pre} \mid k_{\rm pre}).
\end{aligned}\label{eq:r_prev|K_pre}
\end{equation}
The conditional distribution of $R_{\rm pre}$ is obtained from Appendix~\ref{appendix:Helper function} as
\begin{equation}
\begin{aligned}
&\mathcal{P}(R_{\rm pre}=r_{\rm pre} \mid K_{\rm pre}=k_{\rm pre})
\\&~~~~~~~~~~~~=
\Theta(M-1,N,N-1,k_{\rm pre},r_{\rm pre}).
\end{aligned}\label{eq:rprev_theta_G}
\end{equation}
Whenever Eve receives Bob's final innovative packet, it is linearly independent of the previous shared vectors. Therefore, the shared rank at time slot $T$ is
\begin{equation}
R''=R_{\rm pre}+Z.
\label{eq:rsh_G}
\end{equation}

Let $K$ denote the number of packets received only by Eve after Bob reaches full rank. Conditioning on $K$ yields
\begin{equation}
\begin{aligned}
& \mathcal{P}({\bf{s}} \notin \text{row}({\bf{E}}_{T'}) \mid r_{\text{pre}},k_{\text{pre}},z) \\&=\sum_{k=0}^{T'-M}
\mathcal{P}({\bf{s}} \notin \text{row}({\bf{E}}_{T'}) \mid k,r_{\text{pre}},k_{\text{pre}},z)\mathcal{P}(K=k).
\end{aligned}\label{eq:r_prev|k,K_pre}
\end{equation}
The number of packets not received by Bob by time slot $T$ is $T-M$. Also, 
from time slot $T+1$ to $T'$, Alice continues transmitting
linear combinations of the same block, that are all independently generated.  Therefore, the probability that Eve successfully receive $K$ number of these packets 
is
\begin{equation}
\mathcal{P}(K=k)
=
\binom{T'-M}{k}
\text{P}_\text{E}^k(1-\text{P}_\text{E})^{T'-M-k}.
\label{eq:K_extra_G}
\end{equation}
To derive the first term in \eqref{eq:r_prev|k,K_pre}, we span the probability based on the rank of the decoding matrix of Eve by time $T'$ as
\begin{equation}
\begin{aligned}
&\mathcal{P}({\bf{s}} \notin \text{row}({\bf{E}}_{T'}) \mid K,r_{\text{pre}},K_{\text{pre}},z)=\\&\sum_{r'=r''}^{N-1}
\mathcal{P}({\bf{s}} \notin \text{row}({\bf{E}}_{T'}) \mid r',K,r_{\text{pre}},K_{\text{pre}},z)\\&~~~~~\times\mathcal{P}(\text{r}({\bf{E}}_{T'})=r'\mid r'',K)
\end{aligned}\label{eq:r_prev|r',k,K_pre}
\end{equation}

Conditioned on  successfully receiving  $K$  packets by Eve that is not received by Bob, and on the existing shared rank
$r''$, the probability that Eve's final rank becomes
$r'$ is 
\begin{equation}
\mathcal{P}(\text{r}({\bf{E}}_{T'})=r'\mid r'',K)=\Psi(K,N-r'',r'-r''),
\label{eq:Psi_G}
\end{equation}
where $\Psi(K,N-r'',r'-r'')$ denotes the
probability that $K$ random vectors increase the rank by
$r'-r''$ in the quotient space of dimension
$N-r''$ and is defined as 
\begin{equation}
\Psi(m,n,r)=2^{-mn}\Phi(m,n,r),
\label{eq:Psi_def}
\end{equation}
where $\Phi(.,.,.)$ is defined in \eqref{eq:theorem_Phi_def_H}. To derive the first term in \eqref{eq:r_prev|r',k,K_pre}, we have

\begin{equation}
\begin{aligned}
&\mathcal{P}({\bf{s}} \notin \text{row}({\bf{E}}_{T'}) \mid r',K,r_{\text{pre}},K_{\text{pre}},z)\\&=\Pr\!\left(
{\bf s}\notin {\rm row}({\bf E}_{T'})
\mid
{\rm r}({\bf E}_{T'})=r'
\right)
=
\frac{2^N-2^{r'}}{2^N-1}.
\end{aligned}
\label{eq:s_not_row_G}
\end{equation}
where the first equality follows from the fact that
$\Pr({\textbf{s}}\notin\mathrm{row}({\bf E}_{T'}))$
depends only on the rank of ${\bf E}_{T'}$. The second equality
follows since a rank-$r'$ subspace contains $2^{r'}$ vectors
among the $2^N-1$ nonzero vectors in $\mathbb{F}_2^N$.

Combining
\eqref{eq:Psi_G}--\eqref{eq:s_not_row_G}
yields the conditional probability
$\text{P}_{\rm asy|T',T,M}$ in
\eqref{eq:theorem_Pasy_cond_exact_G}. Substituting this result
into \eqref{eq:Pasy_exact_G} gives the exact
asynchronization probability $\text{P}_{\rm asy}$.
Finally, for a deterministic ACK delay,
$f_{\rm ACK}(T'-T)=1$ when
$T'=T+T_{\rm ACK}$ and zero otherwise, reducing
\eqref{eq:Pasy_exact_G} to
\eqref{eq:theorem_Pasy_exact_G} and completing the proof.



\ifCLASSOPTIONcaptionsoff
  \newpage
\fi

\bibliographystyle{IEEEtran}

\bibliography{references.bib}

\end{document}